\documentclass[12pt,twoside,reqno]{amsart}

\usepackage{amssymb,amsmath,amstext,amsthm,amsfonts,amscd,xcolor}

\usepackage{dsfont}

\usepackage[ansinew]{inputenc} 
\usepackage{graphicx}
\usepackage[mathscr]{eucal}
\usepackage{hyperref}

\newcommand{\blob}{\rule[.2ex]{.8ex}{.8ex}}

\newcommand{\R}{\mathbb{R}}
\newcommand{\C}{\mathbb{C}}
\newcommand{\N}{\mathbb{N}}
\newcommand{\Z}{\mathbb{Z}}

\newcommand{\T}{\mathbb{T}}
\newcommand{\SL}{{\rm SL}}

\newcommand{\gl}{{\rm Mat}}
\newcommand{\Mat}{{\rm Mat}}
\newcommand{\V}{\mathscr{V}}

\newcommand{\sabs}[1]{\left| #1 \right|} 
\newcommand{\abs}[1]{\bigl| #1 \bigr|} 
\newcommand{\norm}[1]{\lVert#1\rVert} 
\newcommand{\bnorm}[1]{\Bigl\| #1\Bigr\|} 

\newcommand{\intpart}[1]{\lfloor #1 \rfloor} 
\newcommand{\transl}{{T}} 
\newcommand{\less}{\lesssim}

\newcommand{\ep}{\epsilon}

\newcommand{\cocycles}{\mathcal{C}}
\newcommand{\observables}{\Xi}
\newcommand{\param}{\underline{p}}
\newcommand{\params}{\mathscr{P}}

\newcommand{\gapc}{\kappa}
\newcommand{\nzerobar}{\underline{n_0}}
\newcommand{\nzz}{n_{0 0}}
\newcommand{\nzo}{n_{0 1}}

\newcommand{\mesfs}{\mathscr{I}}
\newcommand{\devfs}{\mathscr{E}}

\newcommand{\dev}{\ep}                
\newcommand{\devf}{\underline{\dev}}  
\newcommand{\mes}{\iota}              
\newcommand{\mesf}{\underline{\mes}}  

\newcommand{\scale}{\mathscr{N}}  

\newcommand{\An}[1]{A^{({#1})}}  
\newcommand{\Bn}[1]{B^{({#1})}}  

\newcommand{\Lan}[2]{L^{({#1})}_{#2}}
 
\newcommand{\comp}{^{\complement}}

\newcommand{\U}{\mathscr{U}}
\newcommand{\Oo}{\mathscr{O}}                
\newcommand{\ind}{\mathds{1}}

\newcommand{\B}{\mathscr{B}}                                                                                                   
\newcommand{\Bbar}{\bar{\mathscr{B}}}

\newcommand{\dist}{{\rm dist}}

\newcommand{\rgap}{{\rm gr}}

\makeatletter
\newsavebox{\@brx}
\newcommand{\llangle}[1][]{\savebox{\@brx}{\(\m@th{#1\langle}\)}%
  \mathopen{\copy\@brx\mkern2mu\kern-0.9\wd\@brx\usebox{\@brx}}}
\newcommand{\rrangle}[1][]{\savebox{\@brx}{\(\m@th{#1\rangle}\)}%
  \mathclose{\copy\@brx\mkern2mu\kern-0.9\wd\@brx\usebox{\@brx}}}
\makeatother

\newcommand{\normtwo}[1]{
{\left\vert\kern-0.25ex\left\vert\kern-0.25ex\left\vert #1 
    \right\vert\kern-0.25ex\right\vert\kern-0.25ex\right\vert} } 

\theoremstyle{plain}
\newtheorem{theorem}{Theorem}[section]
\newtheorem{proposition}{Proposition}[section]
\newtheorem{corollary}[proposition]{Corollary}
\newtheorem{lemma}[proposition]{Lemma}
\newtheorem{definition}{Definition}[section]

\numberwithin{equation}{section}

\newtheorem{remark}{Remark}[section]

\title[Continuity of the Lyapunov exponents]{An abstract continuity theorem of the Lyapunov exponents}

\date{}

\begin{document}

\author[P. Duarte]{Pedro Duarte}
\address{Departamento de Matem\'atica and CMAFIO\\
Faculdade de Ci\^encias\\
Universidade de Lisboa\\
Portugal 
}
\email{pmduarte@fc.ul.pt}

\author[S. Klein]{Silvius Klein}
\address{ Department of Mathematical Sciences\\
Norwegian University of Science and Technology (NTNU)\\
Trondheim, Norway\\ 
and IMAR, Bucharest, Romania }
\email{silvius.klein@math.ntnu.no}

\begin{abstract} We devise an abstract, modular scheme to prove continuity of the Lyapunov exponents  for a general class of linear cocycles. The main assumption is the availability of appropriate large deviation type (LDT) estimates which are uniform in the data.  We provide a modulus of continuity that depends explicitly on the sharpness of the LDT estimate. Our method uses an inductive procedure based on the deterministic, general Avalanche Principle in~\cite{LEbook}. The main advantage of this approach, besides the fact that it provides quantitative estimates, is its versatility, as it applies to quasi-periodic cocycles (one and multivariable torus translations), to random cocycles (i.i.d.  and Markov systems) and  to any other types of base dynamics as long as appropriate LDT estimates are satisfied. Moreover, compared to other available quantitative results for quasi-periodic or random cocycles, this method allows for weaker assumptions.

This is a draft of a chapter  in our forthcoming research monograph~\cite{LEbook}.  
\end{abstract}

\maketitle

\section{Definitions, the abstract setup and statement}
\label{abstract_setup}
An ergodic dynamical system $(X, \mathscr{F}, \mu, T)$ consists of a set $X$, a $\sigma$ - algebra $\mathscr{F}$, a probability measure $\mu$ on $(X, \mathscr{F})$ and a transformation $T \colon X \to X$ which is ergodic and measure preserving. 

Two important classes of ergodic dynamical systems are the shift over a stochastic process (i.e. a Bernoulli shift or a Markov shift) and the torus translation by an incommensurable frequency vector.

A linear cocycle over an ergodic system $(X, \mathscr{F}, \mu, T)$ is a skew-product map on the vector bundle $X \times \R^m$ given by
$$ X \times \R^m \ni (x, v) \mapsto (T x, A (x) v) \in X \times \R^m\,.$$

Hence $T$ is the base dynamics while $A$ defines the fiber action. Since the base dynamics will be fixed throughout this paper, we identify the cocycle with just its fiber action $A$.

The iterates of the cocycle are $(\transl^n \, x , \An{n} (x)\,v)$,  where
$$\An{n} (x)= A(\transl^{n-1} \, x) \cdot  \ldots \cdot A(\transl x) \cdot A(x)\,.$$

The {\em Lyapunov exponents} of the cocycle $A$ are the numbers $L_1 (A) \ge L_2 (A) \ge \ldots \ge L_m (A) > - \infty$,  which measure the average exponential growth of the iterates of the cocycle along invariant fiber subspaces  (see \cite{L.Arnold}). Their multiplicities are correlated with the dimensions of the invariant fiber subspaces. The Lyapunov exponents can also be described, via Kingman's subadditive ergodic theorem, in terms of singular values of iterates of the cocycle as follows:
$$L_k (A) = \lim_{n\to\infty} \int_{\T^d} \frac{1}{n} \, \log s_k [ \An{n} (x)  ] \, d x\,, $$
where $s_1 (g) \ge s_2 (g) \ge \ldots \ge s_m (g) \ge 0$ denote the singular values of a matrix $g \in \Mat (m, \R)$. 
In particular, since the largest singular value $s_1 (g) = \norm{g}$, we have
$$L_1 (A) = \lim_{n\to\infty} \int_{\T^d} \frac{1}{n} \, \log \norm{ \An{n} (x)  } \, d x $$
which is the maximal Lyapunov exponent of $A$ in the sense of Furstenberg.

Given an ergodic system $(X, \mathscr{F}, \mu, T)$, we introduce the main actors - a space of cocycles and a set of observables - and we describe the main assumptions on them - certain uniform large deviation type (LDT) estimates and a uniform $L^p$ - boundedness. Then we formulate an abstract criterion for the {\em continuity} of the corresponding Lyapunov exponents.


\subsection{Cocycles and observables} 
\begin{definition}\label{def:cocycle:space}
A  space of measurable cocycles  $\mathcal{C}$
is any class of matrix valued functions
$A:X\to \Mat_m(\R)$, where $m\in\N$ is not fixed, such that
every  $A:X\to \Mat_m(\R)$ in $\mathcal{C}$ has the following properties:
\begin{enumerate}
\item $A$ is $\mathscr{F}$-measurable,
\item $\norm{A}\in L^\infty(\mu)$,
\item The exterior powers $\wedge_k A: X\to \Mat_{\binom{m}{k}}(\R)$ are in $\mathcal{C}$,
for $k\leq m$.
\end{enumerate}

Each subspace 
$\mathcal{C}_m :=\{\, A\in \mathcal{C} \ | 
A:X\to \Mat_m(\R)\,\} $
is a-priori endowed with a distance 
$\dist \colon \mathcal{C}_m \times \mathcal{C}_m \to [0,+\infty)$ which is at least as fine as the $L^\infty$ distance, i.e. for all $A, B \in \cocycles_m$ we have
$$\dist (B, A) \ge \norm{B-A}_{L^\infty}\,.$$  

We assume a correlation between the distances on each of these subspaces, in the sense that the map 
$$\cocycles_m \ni A \mapsto \wedge_k \, A \in \cocycles_{m \choose k}$$
is locally Lipschitz. 
\end{definition}

\smallskip

Let $A \in \cocycles$ be a measurable cocycle. Since $\norm{A} \in L^\infty$, we have $\log^+ \norm{A} \in L^1$, hence Furstenberg-Kesten's theorem (the non-invertible, one-sided case, see chapter 3 in \cite{L.Arnold}) applies. In particular, if we denote
\begin{equation}\label{def:eq:L1n}
\Lan{n}{1} (A) := \int_X \frac{1}{n} \, \log \norm{ \An{n} (x) } \, \mu (d x)\,,
\end{equation}
then as $n \to \infty$,  $ \Lan{n}{1} (A) \to  L_1 (A)$ (the maximal Lyapunov exponent). 

We call $\Lan{n}{1} (A)$ finite scale (maximal) Lyapunov exponents.

\smallskip

We will need stronger integrability assumptions on the measurable functions $\frac{1}{n} \, \log \norm{ \An{n} (x) }$. 

Let $1 \le p \le \infty$. For simplicity of notations, later on we may assume that $p=2$.

\begin{definition}\label{def:Lp-bounded}
A cocycle $A \in \cocycles$ is called $L^p$ - bounded if there is $C < \infty$, which we call its $L^p$-bound, such that for all $n \ge 1$ we have:
\begin{equation}\label{eq:Lp-bounded}
\bnorm{ \frac{1}{n} \, \log \norm{\An{n} (\cdot)}}_{L^p} < C\,.
\end{equation}
\end{definition}

\begin{definition}\label{def:unif-Lp-bounded}
A cocycle $A \in \cocycles_m$ is called uniformly $L^p$ - bounded if there are $\delta = \delta (A) > 0$ and $C = C(A) < \infty$ such that for all $B \in \cocycles_m$ with $\dist (B, A) < \delta$ and for all $n \ge 1$ we have:
\begin{equation}\label{eq:unif-Lp-bounded}
\bnorm{ \frac{1}{n} \, \log \norm{\Bn{n} (\cdot)}}_{L^p} < C\,.
\end{equation}
\end{definition}

\smallskip

It is not difficult to show that if a cocycle $A \in \cocycles$ satisfies the bounds
$$\norm{ \log \norm{A^{\pm}} }_{L^p} \le C < \infty\,,$$

then for all $n \in \Z$ we have
\begin{equation}\label{log:norm:An}
\norm{ \log \norm{\An{n}} }_{L^p} \le C \abs{n}\,.
\end{equation}

Hence if we assume that
\begin{equation}\label{def:eq:logA^-1Lp}
\log \norm{A^\pm} \in L^p
\end{equation}

holds for all cocycles $A \in \cocycles$, and if we endow $\cocycles_m$  with the distance given by
$$ \dist_p (A, B) := \norm{A-B}_{L^{\infty}} + \norm{ \, \log \norm{A^{-1}} - \log \norm{B^{-1}}}_{L^p} $$
when $1 \le p < \infty$, and by
$$ \dist_\infty (A, B) := \norm{A-B}_{L^{\infty}} \,, $$
when $p = \infty$, then every cocycle $A \in \cocycles$ is {\em uniformly} $L^p$ - bounded.

In the applications we have in this book, the uniform $L^p$- boundedness is automatic.  For instance, in the case of random cocycles, we assume from the beginning the integrability condition~\eqref{def:eq:logA^-1Lp} which implies uniform $L^p$- boundedness.

However, we want that our scheme be applicable also to cocycles that are very singular (i.e. non-invertible everywhere), which is the case of a forthcoming paper on quasiperiodic cocycles, and that is why we make the weaker, uniform $L^p$-boundedness assumption.

\vspace{4pt}

Given a cocycle $A \in \cocycles$ and an integer $N \in \N$, denote by $\mathscr{F}_N (A)$ the algebra generated by the sets $\{ x \in X \colon \norm{\An{n} (x)}  \le c \}$ or $\{ x \in X \colon \norm{\An{n} (x)}  \ge c \}$ where $c \ge 0$ and $0 \le n \le N$.

\smallskip

Let $\observables$ be a set of measurable functions $\xi \colon X \to \R$, which we call observables.

\begin{definition}\label{compatibility_condition}
We say that $\observables$ and $A$ are compatible if for every integer $N \in \N$, for every set $F \in \mathscr{F}_N (A)$ and for every $\epsilon > 0$, there is an observable $\xi \in \observables$ such that:
\begin{equation}\label{compatibility_condition-eq}
\ind_{F} \le \xi \quad \text{and} \quad \int_X \xi \, d \mu \le \mu (F) + \epsilon\,.
\end{equation}
\end{definition}

\subsection{Large deviation type estimates}
As mentioned in the introduction, the main tools in our results are some appropriate large deviations type (LDT) estimates for the given dynamical systems (meaning the base and  the fiber dynamics). A LDT estimate for the base dynamics says that given an observable $\xi \colon X \to \R$, we have
$$\mu \, \{ x \in X \colon \abs{ \frac{1}{n}  \, \sum_{j=0}^{n-1}  \xi (\transl^j x) - \int \xi \, \mu (d x) }  > \dev \} < \mes (n, \dev)\,, $$
where $\dev = o (1)$ and $\mes (n, \dev) \to 0$ (as $n \to \infty$) represent, respectively, the size of the deviation from the mean and the measure of the deviation set. The above inequality should hold for all integers $n \ge n_0 (\xi, \ep)$.

\medskip

In classical probabilities, when dealing with i.i.d. random variables, large deviations are precise asymptotic statements, and the measure of the deviation set decays exponentially. For our purposes, and for the given dynamical systems, we need slightly different types of estimates (not as precise, but for all iterates of the system and satisfying some uniformity properties). Moreover, in some of our applications (e.g. to certain types of quasi-periodic cocycles), the available decay of the measure of the deviation set is not exponential in the number of iterates, but slower than exponential. This is the motivation behind the following formalism.

From now on, $\devf, \, \mesf \colon (0, \infty) \to (0, \infty)$ will represent  functions that describe respectively, the size of the deviation from the mean and the measure of the deviation set. We assume that the deviation size functions $\devf (t)$ are non-increasing. We assume that the deviation set measure functions $\mesf (t)$ are continuous and strictly decreasing to $0$, as $t \to \infty$, at least like a power and at most like an exponential, in other words we assume that:
$$\log t \less \log \frac{1}{\mesf (t)} \less t \quad \text{as } t \to \infty \,.$$ 

Denote by $\phi_{\mesf} (t)$ the inverse of the map $t \mapsto \psi_{\mesf} (t) := t \, [\mesf (t)]^{-1/2}$. We then also assume that the increasing function $\phi_{\mesf} (t)$ does not grow too fast, or more precisely that:
$$\varlimsup_{t \to \infty} \ \frac{\phi_{\mesf} (2t)}{\phi_{\mesf} (t)}  < 2 \,.$$

We will use the notation $\dev_n := \devf (n)$ and $\mes_n := \mesf (n)$ for integers $n$. 

\bigskip

In the applications we have thus far, the deviation size functions are either constant functions $\devf (t) \equiv \ep$ for some $0 < \ep \ll 1$ or powers $\devf (t) \equiv t^{-a}$ for some $a > 0$, while de deviation set measure functions are exponentials $\mesf (t) \equiv e^{- c \, t}$ for some $c > 0$ or sub-exponentials $\mesf (t) \equiv e^{- c \, t^b}$ or $\mesf (t) \equiv e^{- c \, t / (\log t)^b}$.

Let $\devfs$ and $\mesfs$ be some spaces of functions, with $\devfs$ containing deviation size functions $\devf (t)$ and $\mesfs$ containing deviation set measure functions $\mesf (t)$. 
We assume that $\mesfs$ is a convex cone, i.e., the functions $a\cdot\mesf (t)$ and
$\mesf_1 (t)+\mesf_2 (t)$ belong to $\mesfs$ for any $a>0$ and $\mesf_1,\mesf_2\in\mesfs$. 
Let $\params$ be a set of triplets  $\param = (\nzerobar, \devf, \mesf)$, where $\nzerobar \in \N$, $\devf \in \devfs$ and $\mesf \in \mesfs$. An element $\param \in \params$ is called an LDT parameter. Our set of LDT parameters $\params$ should satisfy the condition: for all $\ep > 0$ there is $\param = \param (\ep) = (\nzerobar, \devf, \mesf) \in \params$ such that $\dev_{\nzerobar} \le \ep$, so $\params$ contains LDT parameters with arbitrarily small deviation size functions. 

\medskip


We now define the base and fiber LDT estimates, which are relative to given spaces of deviation functions $\devfs$, $\mesfs$ and set of parameters $\params$. 

\begin{definition}\label{def:base-LDT}
An observable $\xi \colon X \to \R$ satisfies a base-LDT estimate if for every $\ep > 0$ there is $\param = \param (\xi, \ep) \in \params$, $\param = (\nzerobar, \devf, \mesf)$, such that for all $n \ge \nzerobar$ we have $\dev_n \le \ep$ and
\begin{equation}\label{eq:base-LDT}
\mu \, \{ x \in X \colon \abs{ \frac{1}{n} \, \sum_{j=0}^{n-1} \xi (\transl^j x) - \int_X \xi d \mu } > \dev_n \} < \mes_n\,.
\end{equation}
\end{definition}

\begin{definition}\label{def:fiber-LDT}
A measurable cocycle $A \in \mathcal{C}$ satisfies a fiber-LDT estimate if  for every $\ep > 0$ there is $\param = \param (A, \ep) \in \params$, $\param = (\nzerobar, \devf, \mesf)$, such that for all $n \ge \nzerobar$ we have $\dev_n \le \ep$ and
\begin{equation}\label{eq:fiber-LDT}
\mu \, \{ x \in X \colon \abs{ \frac{1}{n} \, \log \norm{ \An{n} (x) } - \Lan{n}{1} (A) } > \dev_n \} < \mes_n\,.
\end{equation}
\end{definition}

\medskip

We will need a stronger form of the fiber-LDT, one that is uniform in a neighborhood of the cocycle, in the sense that estimate~\eqref{eq:fiber-LDT} holds with the same LDT parameter for all nearby cocycles.

\begin{definition}\label{def:unif:fiber-LDT}
A measurable cocycle $A \in \cocycles_m$ satisfies a uniform fiber-LDT  if for all $\ep > 0$ there are $\delta = \delta (A, \ep) > 0$ and $\param = \param (A, \ep) \in \params$, $\param = (\nzerobar, \devf, \mesf)$, such that if  $B \in \cocycles_m$ with $\dist (B, A) < \delta$ and if 
$n \ge \nzerobar$ then $\dev_n \le \ep$ and
\begin{equation}\label{eq:unif-fiber-LDT}
\mu \, \{ x \in X \colon \abs{ \frac{1}{n} \, \log \norm{ \Bn{n} (x) } - \Lan{n}{1} (B) } > \dev_n \} < \mes_n\,.
\end{equation}
\end{definition}


\subsection{The Avalanche Principle} 
\label{ap-subsection}

The scheme described in this paper to derive continuity properties of the Lyapunov exponents of a linear cocycle uses a multiscale analysis argument. This type of argument refers to an inductive procedure on scales (i.e. number of iterates of the system). The inductive step in our method is based upon a deterministic statement, called the Avalanche Principle (AP), relating the norm growth of long products of matrices to the product of norms of individual factors. This principle was first formulated by M. Goldstein and W. Schlag in \cite{GS-Holder}, in the setting of $\SL (2, \C)$ matrices. In Chapter 1 of \cite{LEbook} we prove a general version of the AP (see also our preprint \cite{LEbook-chap2}). We recall here the relevant statement in our version of the AP.

We use the notation $$\rgap (g) = \frac{s_1 (g)}{s_2(g)} \in [1, \infty]$$
for the ratio of the first two singular values of a matrix $g \in \Mat (m, \R)$.

\begin{proposition} \label{AP-practical}
There exists $c>0$ such that
 given $0<\varepsilon<1$,  $0<\varkappa\leq c\,\varepsilon^ 2$ 
and  \,  $g_0, g_1,\ldots, g_{n-1}\in\gl(m,\R)$, \,
 if
\begin{align*}
 & \rgap (g_i) >  \frac{1}{\varkappa} &  \text{for all }  \  0 \le i \le n-1  
\\
 & \frac{\norm{ g_i \cdot g_{i-1} }}{\norm{g_i}  \, \norm{ g_{i-1}}}  >  \varepsilon  & 
 \text{for all }     \ 1 \le i \le n-1  
\end{align*}
then  
\begin{align*} 
&  \sabs{ \log \norm{ g^{(n)} } + \sum_{i=1}^{n-2} \log \norm{g_i} -  \sum_{i=1}^{n-1} \log \norm{ g_i \cdot g_{i-1}} } \less n \cdot \frac{\varkappa}{\varepsilon^2} \;.
\end{align*}
 \end{proposition}

\subsection{Abstract continuity theorem of the Lyapunov exponents}
\label{act_le}

We are ready to formulate the main result of this paper. 

\begin{theorem}\label{abstract:cont_thm}
Consider an ergodic system $(X,\mathscr{F},\mu, \transl)$, a space of measurable cocycles $\mathcal{C}$, a set of observables $\observables$, a set of LDT parameters $\params$ with corresponding spaces of deviation functions $\devfs$, $\mesfs$, let $ 1 < p \le \infty$ and assume the following:
\begin{enumerate}
\item $\observables$ is compatible with every cocycle $A \in \mathcal{C}$.
\item Every observable $\xi \in \observables$ satisfies a base-LDT.
\item Every $A \in \cocycles$ with $L_1 (A) > - \infty$ is uniformly $L^p$ - bounded.
\item Every  cocycle $A \in \mathcal{C}$ for which $L_1 (A) > L_2 (A)$ satisfies a uniform fiber-LDT.
\end{enumerate}

Then all Lyapunov exponents $L_k \colon \cocycles_m \to [ - \infty, \infty)$, $1 \le k \le m$, $m \in \N$ are continuous functions of the cocycle. 

Moreover, given $A \in \cocycles_m$ and $1 \le k \le m$, if $L_{k} (A) > L_{k+1} (A)$, then locally near $A$ the map $L_1 + L_2 + \ldots + L_k$  has a modulus of continuity  
$\omega (h) :=  [ \mesf \ (c \, \log \frac{1}{h}) ]^{1 - 1/p}$ for some $\mesf = \mesf (A) \in \mesfs$ and $c = c (A) > 0$.
\end{theorem}

\vspace{4pt}

The proof of the abstract continuity theorem (ACT) of the Lyapunov exponents will be finalized in Section~\ref{general_cont_thm} and Section~ \ref{modulus_cont}. In Section~\ref{uusc_le} we prove that the upper semicontinuity of the maximal Lyapunov exponent holds uniformly in cocycle and phase, for a large set of phases. While this result is interesting in itself, in our scheme it ensures that in an inductive procedure based on the AP, the  gap condition holds. The inductive procedure, which is a type of multiscale analysis leading to the proof of continuity of the Lyapunov exponents, is described in 
Section~\ref{finite_scale_cont} (the base step) and Section~\ref{inductive_step} (the inductive step). 

We note that the use of the nearly upper semicontinuity of the maximal Lyapunov exponent result in Section~\ref{uusc_le} represents a major point of difference between our inductive procedure and the one employed by M. Goldstein and W. Schlag in \cite{GS-Holder}. It is also what allows us to treat, within this scheme, random models  (see Chapter 5 in ~\cite{LEbook}) and (in a future work) identically singular quasi-periodic models.

\section{Upper semicontinuity of the top Lyapunov exponent}
\label{uusc_le}
Given are an ergodic system $(X,\mathscr{F},\mu, \transl)$, a space of measurable cocycles $\mathcal{C}$, a set of observables $\observables$ and a set of LDT parameters $\params$ with corresponding spaces of deviation functions $\devfs$ and $\mesfs$.



It is well know that the top Lyapunov exponent is upper semicontinuous as a function of the cocycle. Our argument requires a much more precise version of the upper semicontinuity, one that is uniform in the number $n$ of iterates and in the phase $x$. Such results are available, see \cite{JitMavi}, \cite{Furman}, and they are based on a stopping time argument used by Katznelson and Weiss \cite{Katz-Weiss} in their proofs of the Birkhoff's and Kingman's ergodic theorems. However, the results in  \cite{JitMavi}, \cite{Furman} require {\em unique ergodicity} of the system, a property that Bernoulli and Markov shifts do not satisfy.  By replacing unique ergodicity with a weaker property - namely that a base-LDT holds for a large enough set of observables, which we show later to hold for Markov shifts  - we obtain a (weaker) version of the uniform upper semicontinuity in  \cite{JitMavi}, one which holds for a large enough set of phases.


\begin{proposition}[nearly uniform upper semicontinuity]\label{n-unif-usc}
Let $A \in \cocycles_m$ be a measurable cocycle such that $\observables$ and $A$ are compatible
and every observable $\xi \in \observables$ satisfies a base-LDT with corresponding LDT parameter in $\params$.

\begin{enumerate}
\item[(i)] Assume that $L_1 (A) > - \infty$ and that $A$ is $L^1$-bounded. 
\end{enumerate}
For every $\epsilon > 0$, there are $\delta = \delta(A, \epsilon) > 0$, $n_0 = n_0 (A, \epsilon) \in \N$ and $\mesf = \mesf (A, \ep) \in \mesfs$, such that if $B \in \mathcal{C}_m$ with $d (B, A) < \delta$, and if $n \ge n_0$, then the upper bound
\begin{equation}\label{n-unif-usc-eq}
\frac{1}{n} \, \log \norm{B^{(n)} (x) } \le L_1 (A) + \epsilon
\end{equation}
holds for all $x$ outside of a set of measure $ < \mes_n$.

Up to a zero measure set, the exceptional set depends only on $A$, $\ep$.

\vspace{2mm}

\begin{enumerate}
\item[(ii)] Assume that $L_1 (A) = - \infty$.
\end{enumerate}
For every $t < \infty$, there are $\delta = \delta(A, t) > 0$, $n_0 = n_0 (A, t) \in \N$ and $\mesf = \mesf (A, t) \in \mesfs$, such that if $B \in \mathcal{C}_m$ with $d (B, A) < \delta$, and if $n \ge n_0$, then the upper bound
\begin{equation}\label{n-unif-usc-eq-infty}
\frac{1}{n} \, \log \norm{B^{(n)} (x) } \le -  t
\end{equation}
holds for all $x$ outside of a set of measure $ < \mes_n$.


Up to a zero measure set, the exceptional set depends only on $A$, $t$.
\end{proposition}

\begin{proof} Throughout this proof, $C$ will stand for a positive, finite, large enough constant that depends only on the cocycle $A$, and which may change slightly from one estimate to another. 

If $B \in \cocycles$ is at some small distance from $A$, then it will be close enough to $A$ in the $L^\infty$ distance as well, hence we will assume that for $\mu$ a.e. $x \in X$ we have $\norm{B (x) } < C$.

Moreover, in the case {\rm (i)}, when $L_1 (A)$ is finite, since we also assume $A$ to be $L^1$ bounded, we may choose the constant $C$ such  that for all $n \ge 1$ we have 
$ \bnorm{ \frac{1}{n} \, \log \norm{\An{n} (\cdot)}}_{L^1} < C$ and hence also $\abs{ L_1 (A) } < C$.
 
\medskip

The proofs for each of the two cases are similar, but the argument will differ in some parts. We first present the case $L_1 (A) > - \infty$ in detail, then indicate how to modify the argument for the case $L_1 (A) = - \infty$.

\smallskip

{\rm (i)}  Fix $\ep > 0$. By Kingman's subadditive ergodic theorem, 
$$ \lim_{n \to \infty} \, \frac{1}{n} \log \norm{\An{n} (x)} = L_1 (A) \quad \text{ for } \mu \text{ a.e. } x\,,$$
hence the number
\begin{equation}\label{def-n(x)}
n (x) := \min \{ n \ge 1 \colon \frac{1}{n} \log \norm{\An{n} (x)} < L_1 (A) + \ep \}
\end{equation} 
is defined for $\mu$ a.e. $x \in X$. 

For every integer $N$, let 
$$\U_N := \{ x \colon n (x) \le N \} = \bigcup_{n=1}^N \{ x \colon  \frac{1}{n} \log \norm{\An{n} (x)} < L_1 (A) + \ep \}\,.
$$
Then $\U_N\comp \in \mathscr{F}_N (A)$, $\U_N \subset \U_{N+1}$ and $\cup_N \,  \U_N$ has full measure. Therefore, there is $N = N (\ep, A)$ such that $\mu (\U_N\comp) < \ep$. 

We fix this integer $N$ for the rest of the proof and denote the set $\U = \U (\ep, A) : = \U_N$. Therefore,  $\U\comp \in \mathscr{F}_N (A)$, $\mu (\U\comp) < \ep$ and we have: if $x \in \U$ then $ 1 \le n (x) \le N$ and 
\begin{equation}\label{unif-usc-eq1}
\log \norm{\An{n(x)} (x) } \le n (x) L_1 (A) + n (x) \ep\,.
\end{equation}

Next we will bound from above $\log \norm{\Bn{n} (x) }$ by $\log \norm{ \An{n} (x)} + o (1)$ for all cocycles $B$ with $\dist (B, A) < \delta$ where $\delta$ will be chosen later, for all $1 \le n \le N$ and for a large set of phases $x \in X$.

Since $A$ is $L^1$-bounded,  $\log \norm{\An{n}} \in L^1 (X, \mu)$, so $\An{n} (x) \neq 0$ for $\mu$ - a.e. $x \in X$.

Moreover, if $B \in \mathcal{C}_m$ with $\dist (B, A) < \delta$ (where $\delta \ll 1$ is chosen below), we have $\norm{B (x) - A (x)} < \delta$ and $\norm{B (x)} < C$ for $\mu$ - a.e. $x \in X$.

Then for $x$ outside a null set and for $1 \le n \le N$, we have:
\begin{align*}
& \log \norm{\Bn{n} (x) } - \log \norm{\An{n} (x) } = \log \frac{ \norm{\Bn{n} (x) } }{ \norm{\An{n} (x) } }   \\
& \qquad \leq \log [  \frac{ \norm{\Bn{n} (x) - \An{n} (x)} }{ \norm{\An{n} (x) }} + 1 ]   \le  \frac{ \norm{\Bn{n} (x) - \An{n} (x)} }{ \norm{\An{n} (x) }}    \\  
& \qquad \leq n C^{n-1} \delta \cdot \frac{1}{\norm{\An{n} (x)}} \le N C^{N-1} \delta \cdot \frac{1}{\norm{\An{n} (x)}}\;.
\end{align*} 
Hence
\begin{equation}\label{unif-usc-eq2}
\log \norm{\Bn{n} (x) } \le  \log \norm{\An{n} (x) } + \delta \, N C^{N-1} \, \frac{1}{\norm{\An{n} (x)}} \
\end{equation}
for all $x$ outside a zero measure set and for all $1 \le n \le N$.

Let $t := e^{-N^2 C /\ep}$. Consider the set
$\V := \bigcap_{n=1}^{N} \ \{ x \colon \norm{A^{(n)} (x)} > t \} $. Clearly $\V\comp \in \mathscr{F}_N (A)$, and we will show that $\V\comp$ has measure at most $\ep$.

\smallskip

If for some $1 \le n \le N$ and $x \in X$ we have $\norm{\An{n} (x)} \le t \, ( <1)$, then 
$$\abs{\frac{1}{n} \, \log \norm{A^{(n)} (x)} } > \frac{\log 1/t}{n}\,,$$
hence
$$\V\comp  \subset \bigcup_{n=1}^N \,  \{x \colon \abs{\frac{1}{n} \, \log \norm{A^{(n)} (x)} } > \frac{\log 1/t}{n} \} \,.$$

Since $A$ is $L^1$-bounded, there is $C = C(A) < \infty$ such that for all $n \ge 1$
$$\bnorm{ \frac{1}{n} \, \log \norm{\An{n} (\cdot)}}_{L^1} < C\,.$$

Then by Chebyshev's inequality,
$$\mu \ \{x \colon \abs{\frac{1}{n} \, \log \norm{A^{(n)} (x)} } > \frac{\log 1/t}{n} \} < \frac{C n}{\log 1/t} \le \frac{C N}{\log 1/t} = \frac{\ep}{N}\,.$$

Therefore,
$$\mu (\V\comp) < N \, \ep/N= \ep\,,$$ 
and if $1 \le n \le N$ then for $\mu$ a.e. $x \in \V$ we have
$$\log \norm{\Bn{n} (x) } \le  \log \norm{\An{n} (x) } + \delta \, N C^{N-1} \, e^{N^2 C /\ep} < \log \norm{\An{n} (x) } + \ep\,,$$
provided we choose  $\delta < \delta (\ep, C, N) = \delta (\ep, A)$ small enough.

Let $\Oo := \U \cap \V$. Then $\Oo\comp \in \mathscr{F}_N (A)$ and $\mu (\Oo\comp) < 2 \ep$. 
We conclude that  for $\mu$ almost every $x \in \Oo$,  we have:
\begin{equation}\label{unif-usc-star}
\log \norm{\Bn{n(x)} (x) } \le n (x) \, L_1 (A) + n(x) \, 2 \ep\,.
\end{equation} 

Let $n_0 = n_0 (\ep, A) := \frac{C N}{\ep}$.

Fix $x \in X$ and define inductively for all $k \ge 1$ the sequence of phases $x_k = x_k (x) \in X$ and the sequence of integers $n_k = n_k (x) \in \N$ as follows:

\begin{align*}
x_1 &= x & & n_1 = 
\begin{cases}
n(x_1) & \text { if } x_1 \in \Oo\\
1  & \text{ if } x_1 \notin \Oo
\end{cases} \\
x_2 &= T^{n_1} x_1 & & n_2 = 
\begin{cases}
n(x_2) & \text { if } x_2 \in \Oo\\
1  & \text{ if } x_2 \notin \Oo
\end{cases} \\
\ldots \\
x_{k+1} &= T^{n_k} x_k & & n_{k+1} = 
\begin{cases}
n(x_{k+1}) & \text { if } x_{k+1} \in \Oo\\
1  & \text{ if } x_{k+1} \notin \Oo\,.
\end{cases} 
\end{align*} 

Note that for all $k \ge 1$, $x_{k+1} = T^{n_k + \ldots + n_1}  \, x$ and $ 1 \le n_k \le N$.

\smallskip

For any $n \ge n_0 (> N \ge n_1)$, there is $p \ge 1$ such that
$$n_1 + \ldots + n_p \le n < n_1 + \ldots + n_p + n_{p+1}\,,$$
so $n = n_1 + \ldots + n_p + m$, where $ 0 \le m < n_{p+1} \le N$.

For any cocycle $B$ such that $d (B, A) < \delta$, let $b_n (x) := \log \norm{\Bn{n} (x)}$. Then clearly for $\mu$-a.e. $x \in X$ we have $b_n (x) \le n \, C$, where $C$ is a constant that depends on $A$ and 
$b_n (x)$ is a sub-additive process, meaning:
$$b_{n+m} (x) \le b_n (x) + b_m (T^n x)$$
for all $n, m \ge 1$ and for $\mu$ almost every  $x \in X$.

Using this sub-additivity and the definition of $x_k (x), n_k (x)$, we have:
\begin{equation}\label{unif-usc-subadd}
\log \norm{\Bn{n} (x)} = b_n (x) = b_{n_1 + \ldots + n_p + m} (x) \le \sum_{k=1}^{p} b_{n_k} (x_k) + b_m (x_{p+1})\,.
\end{equation} 

We estimate each term separately. Each estimate is valid for $x$ outside of a null set.

For the last term we use the trivial bound:
\begin{equation}\label{unif-usc-last}
b_m (x_{p+1}) \le m C < N C\,.
\end{equation}

For every $1 \le k \le p$ we have:

$\blob$ Either $x_k \in \Oo$, so $n_k = n (x_k)$, in which case, using \eqref{unif-usc-star} we get:
\begin{align*}
b_{n_k} (x_k) & = \log \norm{\Bn{n(x_k)} (x_k)} \le n(x_k) \, L_1 (A) + 2 \ep \, n(x_k)  \\ &= (L_1 (A) + 2 \ep)  \, n_k\,.
\end{align*}

$\blob$ Or  $x_k \notin \Oo$, so $n_k = 1$, in which case $b_{n_k} (x_k) = \log \norm{B (x_k)} \le C$.

Therefore, 
\begin{align*}
b_{n_k} (x_k) & = b_{n_k} (x_k) \, \ind_{\Oo} (x_k) + b_{n_k} (x_k) \, \ind_{\Oo\comp} (x_k) \notag\\ 
& \le (L_1 (A) + 2 \ep) \, n_k  \, \ind_{\Oo} (x_k)  + C \,  \ind_{\Oo\comp} (x_k) \\
& =  (L_1 (A) + 2 \ep) \, n_k - (L_1 (A) + 2 \ep) \, n_k  \, \ind_{\Oo\comp} (x_k)  + C \,  \ind_{\Oo\comp} (x_k) \\
& =  (L_1 (A) + 2 \ep) \, n_k  - (L_1 (A) + 2 \ep)   \, \ind_{\Oo\comp} (x_k)  + C \,  \ind_{\Oo\comp} (x_k) \,,
\end{align*}
where in the last equality we used the fact that $n_k = 1$ when $x_k \in \Oo\comp$.

We conclude:
\begin{align}
b_{n_k} (x_k) & \le (L_1 (A) + 2 \ep) \, n_k + (C - L_1 (A) - 2 \ep)  \, \ind_{\Oo\comp} (x_k)  \notag \\
& < (L_1 (A) + 2 \ep) \, n_k + 3 C \, \ind_{\Oo\comp} (x_k) \,. \label{unif-usc-first}
\end{align}

We add up \eqref{unif-usc-last} and \eqref{unif-usc-first} for all $1\le k \le p$, and then use \eqref{unif-usc-subadd} to get:
\begin{align*}
\log \norm{\Bn{n} (x)} & \le (n_1 + \ldots + n_p) \, (L_1 (A) + 2 \ep) + 3 C \, \sum_{k=1}^{p}  \ind_{\Oo\comp} (x_k)  + C N\\
& \le n \,   (L_1 (A) + 2 \ep) + 3 C \, \sum_{j=0}^{n-1}  \ind_{\Oo\comp} (T^j \, x) + C N\,.
\end{align*}

Divide both sides by $n$ to conclude that for $\mu$-a.e. $x \in X$ and for all $n \ge n_0$,
\begin{equation}\label{unif-usc-eqf1}
\frac{1}{n} \, \log \norm{\Bn{n} (x)} \le L_1 (A) + 2 \ep + 3 C \, \frac{1}{n} \sum_{j=0}^{n-1}  \ind_{\Oo\comp} (T^j \, x)  + \frac{C N}{n}\,.
\end{equation}

By the choice of $n$ we have $\frac{C N}{n} < \ep$, so all is left is to estimate the Birkhoff average above. We use the compatibility condition. Since $\Oo\comp \in \mathscr{F}_N (A)$, there is an observable $\xi = \xi (A, \ep) \in \observables$ such that $\ind_{\Oo\comp} \le \xi$ and $\int_X \xi \, d \mu < \mu (\Oo\comp) + \ep < 3 \ep$.
 Then, applying the base-LDT to $\xi$, there is $\param = \param (\xi, \ep) = \param(A, \ep) \in \params$, $\param = (\nzerobar, \devf, \mesf)$, such that for $n \ge \nzerobar$  we have $\dev_n \le \ep$ and 
 $$ \frac{1}{n} \sum_{j=0}^{n-1}  \ind_{\Oo\comp} (\transl^j \, x) \le  \frac{1}{n} \sum_{j=0}^{n-1}  \xi (\transl^j \, x) < \int_X \xi \, d \mu + \dev_n < 4 \ep\,,$$
 provided we choose $x$ outside a set of measure $\mes_n$. 
 
 This ends the proof in the case $L_1 (A) > - \infty$.
 
 
 {\rm (ii)} The case $L_1 (A) = - \infty$. Let $t$ be large enough, say $t > C + 1$. We apply again Kingman's subadditive theorem and for $\mu$ a.e. $x \in X$, define the integers
 \begin{equation}\label{def-n(x)-infty}
n (x) := \min \{ n \ge 1 \colon \frac{1}{n} \log \norm{\An{n} (x)} < - 2 t \}\,.
\end{equation} 

The sets $\U_N$ are defined as before. Fix $N = N (A, t)$, then $\U = \U (A, t) = \U_N$ so that $\mu (\U\comp) < 1/t$. Furthermore, $\U\comp \in \ \mathscr{F}_N (A)$ and
if $x \in \U$ then $ 1 \le n (x) \le N$ and 
\begin{equation}\label{unif-usc-eq1-infty}
\log \norm{\An{n(x)} (x) } \le - 2 t n (x)\,.
\end{equation}

We show that \eqref{unif-usc-eq1-infty} holds also for cocycles $B$ in a neighborhood of $A$. This is where the argument differs from the case $L_1 (A) > - \infty$.

Let $0 < \delta < \frac{e^{- 2 N t}}{t \, N \, C^{N-1}}$, and let $B \in \cocycles_m$ with $\dist (B, A) < \delta$, so $\norm{B (x) - A(x)} < \delta$ for $\mu$-a.e. $x \in X$.
Then clearly, for any $1 \le m \le N$ and for $\mu$-a.e. $x \in X$ we have:
\begin{align}\label{eq10-usc}
\norm{ \Bn{m} (x) - \An{m} (x)  } < m \, C^{m-1}  \delta \le N \, C^{N-1} \delta < \frac{e^{- 2 N t}}{t} \le \frac{e^{- 2 m t}}{t}\,.
\end{align}

For these phases $x$ and number of iterates $m$, there are two cases.

\smallskip

\textit{Case 1.}  $\norm{\An{m} (x) } $ is extremely small, i.e.
$$\norm{\An{m} (x) }  < e^{ - 2 t \, m}.$$

In this case, using \eqref{eq10-usc} we get $ \norm{\Bn{m} (x) } <  e^{ - 2 t \, m} \, (1 + 1/t)$, so
\begin{equation}\label{eq11-usc}
\log \, \norm{\Bn{m} (x) } < - 2 t \, m + 1/t\,.
\end{equation}

\textit{Case 2.}  $\norm{\An{m} (x) } $ has a lower bound:
 $$\norm{\An{m} (x) }  \ge e^{ - 2 t \, m}\,.$$

Then using again \eqref{eq10-usc}, we have
\begin{align*}
\log \,  \norm{\Bn{m} (x) }  - \log \,  \norm{\An{m} (x) } 
& \le  \frac{ \norm{\Bn{m} (x) - \An{m} (x)} }{ \norm{\An{m} (x) }} & \\
& \le \frac{e^{- 2 m t}}{t} \, e^{2 t \, m} = 1/t\,,
\end{align*}
hence
\begin{equation}\label{eq12-usc}
\log \, \norm{\Bn{m} (x) } < \log \, \norm{\An{m} (x) } + 1/t\,.
\end{equation}

If $m = n(x)$, then using \eqref{eq12-usc} in the second case and using directly \eqref{eq11-usc} in the first case, we conclude that for $\mu$-a.e. $x \in \U$ we have
\begin{equation}\label{unif-usc-star-infty}
\log \norm{\Bn{n(x)} (x) } \le - 2 t \, n (x) + 1/t \le (- 2 t + 1/t) \, n (x) \,,
\end{equation} 
which is the analogue of \eqref{unif-usc-star} when $L_1 (A) > - \infty$.

The rest of the proof then follows exactly the same pattern as when $L_1 (A) > - \infty$, the role of $L_1 (A) + 2 \ep$ being now played by $- 2 t + 1/t$, while the small set $\Oo$ is simply $\U$, since there was no extra small set excluded when deriving \eqref{unif-usc-star-infty}.
 \end{proof}

\begin{remark}\normalfont
Note that since our cocycles are in $L^\infty$, Proposition~\ref{n-unif-usc} above also implies the upper semicontinuity of the top Lyapunov exponent as a function of the cocycle. In particular, this gives continuity at cocycles $A$ with $L_1 (A) = - \infty$, and since $L_1 (A) \ge L_2 (A) \ge \ldots \ge L_m (A)$, this implies that every Lyapunov exponent is continuous at $A$.
Therefore, from now on, we may assume that $L_1 (A) > - \infty$.
\end{remark}

\begin{remark}\normalfont
If $(X, \mu, T)$ is uniquely ergodic (e.g. an ergodic  torus translation), we may choose $\observables$ to be the set of all continuous functions on $X$. Then using Urysohn's lemma, it is easy to verify that the compatibility condition between $\observables$ and $A$ holds. For uniquely ergodic systems, the convergence in Birkhoff's ergodic theorem is uniform in the phase for all continuous observables. Hence the base-LDT estimate~\eqref{def:base-LDT} holds automatically for all $\xi \in \observables$, with deviation measure function $\mesf (t) \equiv 0$. 


It follows that if $(X, \mu, T)$ is uniquely ergodic, then the statements in Proposition~\ref{n-unif-usc} above hold for a.e. $x$. If we assume that the cocycles are continuous, they hold for all phases $x$. Hence we recover the corresponding result in \cite{JitMavi}.
\end{remark} 

The main application of Proposition~\ref{n-unif-usc} is the following lemma, which we use repeatedly throughout the inductive argument. It gives us a lower bound on the gap between the first two singular values of the iterates of a cocycle, thus ensuring the gap condition in the Avalanche Principle.

Throughout this paper, if $A \in \cocycles$ is such that $L_1 (A) > L_2 (A) \ge - \infty$, then $\gapc (A)$ denotes the gap between the first two LE, i.e. $\gapc (A) := L_1 (A) - L_2 (A) > 0$ when $L_2 (A) > - \infty$, while if $L_2 (A) = - \infty$ then $\gapc (A)$ is a fixed, large enough {\em finite} constant.

\begin{lemma}\label{lemma:usc-gapn}
Let $A \in \cocycles_m$ be a cocycle for which $L_1 (A) > L_2 (A)$ and let $\ep > 0$. There are $\delta_0 = \delta_0 (A,  \ep) > 0$, $n_0 = n_0 (A, \ep) \in \N$ and $\mesf = \mesf (A, \ep) \in \mesfs$ such that for all $B \in \cocycles_m$ with $\dist (B, A) < \delta_0$ and for all $n \ge n_0$, if 
\begin{equation}\label{eq1:usc-gapn}
\abs{ \Lan{n}{1} (B) - \Lan{n}{1} (A) } < \theta,
\end{equation}
then for all phases $x$ outside a set of measure $ < \mes_n$ we have:
\begin{equation}\label{eq2:usc-gapn}
\frac{1}{n} \, \log \rgap (\Bn{n} (x) ) > \gapc (A) -2 \theta - 3 \ep\,.
\end{equation}
Moreover,
\begin{equation}\label{eq3:usc-gapn}
\Lan{n}{1} (B) - \Lan{n}{2} (B) > (\gapc (A) -2 \theta - 3 \ep) \, (1 - \mes_n)\,.
\end{equation}
\end{lemma}

\begin{proof}

Fix $\ep > 0$. If $L_2 (A) = - \infty$, let $t = t (A) := -2 L_1 (A) + \gapc (A)$.

Since $L_1(A) > L_2 (A)$, the cocycle $A$ satisfies a uniform fiber-LDT with a parameter $\param = \param (A, \ep) \in \params$ and in a neighborhood around $A$ of size $\delta (A, \ep) >0$.

\smallskip

The compatibility condition holds for all cocycles in $\cocycles$, hence also for $\wedge_2 A$. 
Note that $L_1 (\wedge_2 A) = L_1 (A) + L_2 (A)$, hence $L_1 (\wedge_2 A) > - \infty$ iff $L_2 (A) > - \infty$.

The nearly uniform upper semicontinuity of the top LE (Proposition~\ref{n-unif-usc}) can then be applied to $\wedge_2 A$, and it gives parameters $\delta > 0$, $\mesf \in \mesfs$, $n_0 \in \N$ that define the range of validity of \eqref{n-unif-usc-eq} and \eqref{n-unif-usc-eq-infty} respectively. These parameters depend on $A$ and $\ep$ when $L_2 (A) > - \infty$ and only on $A$ when $L_2 (A) = - \infty$.  

Pick $\delta = \delta (A, \ep) > 0$, $n_0 = n_0 (A, \ep) \in \N$, $\mesf = \mesf (A, \ep) \in \mesfs$ such that both the uniform fiber-LDT and Proposition~\ref{n-unif-usc} apply for all cocycles $B \in \cocycles_m$ with $\dist (B, A) < \delta$, for all $n \ge n_0$ and for all $x$ outside a set of measure $ < \mes_n$. Fix such $B$, $n$, $x$.

For any matrix $g \in \Mat (m, \R)$ we have
\begin{equation}\label{eq100-usc}
\rgap (g) = \frac{s_1 (g)}{s_2 (g)} = \frac{\norm{g}^2}{\norm{\wedge_2 g}} \in [1, \infty]\,.
\end{equation}


From \eqref{eq100-usc} we get
\begin{equation}\label{eq4:usc-gapn}
\frac{1}{n} \, \log \rgap (\Bn{n} (x)) = 2 \, \frac{1}{n} \, \log \norm{\Bn{n} (x)} -  \frac{1}{n} \, \log \norm{\wedge_2 \Bn{n} (x)}\,.  
\end{equation} 

The uniform fiber-LDT gives a lower bound on the first term on the right hand side of  \eqref{eq4:usc-gapn}:
\begin{equation*}
\frac{1}{n} \, \log \norm{\Bn{n} (x)} > \Lan{n}{1} (B) - \dev_{n} > \Lan{n}{1} (B) - \ep\,. 
\end{equation*}
Moreover, from assumption \eqref{eq1:usc-gapn} we have
\begin{equation*}
\Lan{n}{1} (B) > \Lan{n}{1} (A) - \theta \ge L_1 (A) - \theta,
\end{equation*}
hence 
\begin{equation}\label{eq5:usc-gapn}
\frac{1}{n} \, \log \norm{\Bn{n} (x)} >  L_1 (A) - \theta - \ep\,. 
\end{equation}

Proposition~\ref{n-unif-usc} applied to the cocycle $\wedge_2 A$ will give an upper bound on $\frac{1}{n} \, \log \norm{\wedge_2 \Bn{n} (x)}$.

If $L_2 (A) > - \infty$, so $L_1 (\wedge_2 A) > - \infty$, from part {\rm (i)} of Proposition~\ref{n-unif-usc}  we get
\begin{equation}\label{eq6:usc-gapn}
\frac{1}{n} \, \log \norm{\wedge_2 \Bn{n} (x)} < L_1 (\wedge_2 A) + \ep =  L_1 (A) + L_2 (A) + 2 \ep\,.
\end{equation}

Combine \eqref{eq4:usc-gapn}, \eqref{eq5:usc-gapn}, \eqref{eq6:usc-gapn} to conclude that for all chosen $B$, $n$, $x$ we have:
\begin{equation*}\label{pointwise-gap:n0}
\frac{1}{n} \, \log \rgap (\Bn{n} (x)) > \gapc (A) - 2 \theta - 3 \ep\,,
\end{equation*}
which proves \eqref{eq2:usc-gapn}. Integrating in $x$ we derive  \eqref{eq3:usc-gapn}.


Now if $L_2 (A) = - \infty$, so $L_1 (\wedge_2 A) = - \infty$, use part {\rm (ii)} of Proposition~\ref{n-unif-usc}  to get
\begin{equation}\label{eq200:usc-gapn}
\frac{1}{n} \, \log \norm{\wedge_2 \Bn{n} (x)} < - t = 2 L_1 (A) - \gapc (A)\,.
\end{equation}

Combine  \eqref{eq4:usc-gapn}, \eqref{eq5:usc-gapn}, \eqref{eq200:usc-gapn} and get \eqref{eq2:usc-gapn} in this case as well.
Then \eqref{eq3:usc-gapn} follows as above.
 \end{proof}

For the rest of this paper, we are given an ergodic system $(X,\mathscr{F},\mu, \transl)$, a space of measurable cocycles $\mathcal{C}$, a set of observables $\observables$ and a set of LDT parameters $\params$ with corresponding spaces of deviation functions $\devfs$ and $\mesfs$. We assume the compatibility condition~\ref{compatibility_condition} between $\observables$ and any cocycle $A \in \cocycles$, the base-LDT for any observable $\xi \in \observables$, the uniform $L^p$-boundedness condition (put $p=2$ for simplicity of notation) on any cocycle $A \in \cocycles$ with $L_1 (A) > - \infty$ and the uniform fiber-LDT for any cocycle $A \in \cocycles$ with $L_1 (A) > L_2 (A)$. These LDT estimates hold for parameters $\param \in \params$. 

\section{Finite scale continuity}
\label{finite_scale_cont}
We show that the finite scale Lyapunov exponents have a continuous behavior if the scale is {\em fixed}. 
We are not able to prove actual continuity of these finite scale quantities, unless we make some restrictions on the space of cocycles. However, this continuous behavior at finite scale is sufficient for our purposes, as the inductive procedure described in the next section leads to the actual continuity of the limit quantities (the LE) as the scale goes to infinity.

\begin{proposition}[finite scale uniform continuity]\label{prop:finite-scale}
Let $A \in \cocycles_m$ be a cocycle for which $L_1 (A) > L_2 (A)$. There are $\delta_0 = \delta_0 (A) >0$, $\nzo = \nzo (A)$, $C_1 = C_1 (A) > 0$ and $\mesf = \mesf (A) \in \mesfs$ such that for any two cocycles $B_1, B_2 \in \cocycles_m$ with $\dist (B_i, A) \le \delta_0$ where  $ i=1, 2$, if $n \ge \nzo$ and $\dist (B_1, B_2) < e^{- C_1 \, n}$, then
\begin{equation}\label{eq:finite-scale}
\abs{ \Lan{n}{1} (B_1) - \Lan{n}{1} (B_2) } < \mes_n^{1/2}.
\end{equation}

\end{proposition}

\begin{proof}
Let $\ep_0 := \gapc (A)/10 > 0$. Since $L_1 (A) > L_2 (A)$, the uniform fiber-LDT and Lemma~\ref{lemma:usc-gapn} hold for $A, \ep_0$. Choose parameters $\param = \param (A) \in \params$, $\param = (\nzerobar, \devf, \mesf)$ and $\delta_0 = \delta_0 (A) >0$ such that $\dev_{\nzerobar} \le \ep_0$ and Lemma~\ref{lemma:usc-gapn} and the fiber-LDT hold for all cocycles $B \in \cocycles_m$ with $\dist (B, A) \le \delta_0$ and for all $n \ge \nzerobar$. 

Let $C_0 = C_0 (A) > 0$ such that for all such cocycles $B$ we have $\norm{B}_{L^\infty} \le e^{C_0}$ and for all $n \ge 1$, 
$$\abs{\Lan{n}{1} (B)} \le  \bnorm{ \frac{1}{n} \, \log \norm{\Bn{n} (x)} }_{L^2} \le C_0\,.$$

Pick  $C_1 > 2 C_0 + \ep_0$ and $\nzo \ge \nzerobar$ such that $e^{- C_1 \, \nzo} < \delta_0$. 

Let $n \ge \nzo$ and $B_i \in \cocycles_m$ with $\dist (B_i, A) \le \delta_0$ ($i = 1, 2$) be arbitrary but fixed. Assume that $\dist (B_1, B_2) < e^{ - C_1 \, n}$.

Apply the fiber-LDT to each $B_i$ and conclude that for all $x$ outside a set $\B_n^i$, with $\mu (\B_n^i)  < \mes_n$ we have:
\begin{equation}\label{eq1:finite-scale}
\frac{1}{n} \, \log \norm{\Bn{n}_i (x) } > \Lan{n}{1} (B_i) - \ep_n \ge -C_0 - \ep_0\,.
\end{equation}
Let $\B_n = \B_n^1 \cup \B_n^2$, so $\mu (\B_n) < 2 \, \mes_n$ and if $x \in \B_n\comp$ then  we have:
\begin{equation}\label{eq10:finite-scale}
\norm{\Bn{n}_1 (x)}, \ \norm{\Bn{n}_2 (x)} > e^{- (C_0 + \ep_0) \, n} \,.
\end{equation}

Moreover, for $\mu$ - a.e. $x \in \B_n\comp$ and all $0 \le j \le n-1$ we also have:
$$\norm{ B_1 (\transl^j  \, x) - B_2 (\transl^j \, x) } \le \dist (B_1, B_2) < e^{-  C_1 \, n} \,.$$  

Therefore, for $\mu$ - a.e. $x \in \B_n\comp$ we get:
\begin{align*}
\abs{ \frac{1}{n} &\log \norm{B_1^{(n)}(x)} - \frac{1}{n}\log \norm{B_2^{(n)}(x)}  }  = 
\frac{1}{n} \abs{ \log \frac{ \norm{B_1^{(n)}(x)} }{ \norm{B_2^{(n)}(x)} } } \\
& \leq \frac{1}{n} \ \frac{ \norm{B_1^{(n)}(x)  -  B_2^{(n)}(x)}  }{ \min \{  \norm{B_1^{(n)}(x)}, \norm{B_2^{(n)}(x)} \}  }\\
& \leq \frac{1}{n} \sum_{j=0}^{n-1}
e^{ (C_0 + \ep_0) \,n} \ \norm{B_2^{(n-j-1)}( T^ {j+1} x)}\,\norm{B_1(T^j x)-B_2(T^ j x)}\,\norm{ B_1^{(j)}(x)} \\
& \leq \frac{1}{n} \sum_{j=0}^{n-1}
e^{ (C_0 + \ep_0) \,n}  \, e^{C_0 \, (n - j -1)} \ e^{- C_1 \, n} \ e^{C_0 \, j}  
 \le  e^{- n \, (C_1 - 2 C_0 - \ep_0) } \,.
\end{align*}

Integrating in $x$ we conclude:
\begin{equation}\label{eq3:finite-scale}
\int_{\B_n\comp} \ \abs{ \frac{1}{n} \log \norm{B_1^{(n)}(x)} - \frac{1}{n}\log \norm{B_2^{(n)}(x)}  }  \, \mu (d x) <  e^{- n \, (C_1 - 2 C_0 - \ep_0)} \,.
\end{equation}

By Cauchy-Schwarz we have
\begin{align*}
& \int_{\B_n} \ \abs{ \frac{1}{n} \log \norm{B_1^{(n)}(x)} - \frac{1}{n}\log \norm{B_2^{(n)}(x)}  }  \, \mu (d x)  \le \\ 
& \qquad \bnorm{ \frac{1}{n} \, \log \norm{\Bn{n}_1 (x) }  }_{L^2} \cdot \mu (\B_n)^{1/2} + \bnorm{ \frac{1}{n} \, \log \norm{\Bn{n}_2 (x) }  }_{L^2} \cdot \mu (\B_n)^{1/2} \,,
\end{align*}
hence
\begin{equation}\label{eq4:finite-scale}
\int_{\B_n} \ \abs{ \frac{1}{n} \log \norm{B_1^{(n)}(x)} - \frac{1}{n}\log \norm{B_2^{(n)}(x)}  }  \, \mu (d x) \less C_0 \, \mes_n^{1/2} \,.
\end{equation}
Since $\mesf \in \mesfs$ decays at most exponentially, we may of course assume that  $e^{- n \, (C_1 - 2 C_0 - \ep_0)} < \mes_n^{1/2}$, so \eqref{eq3:finite-scale} and \eqref{eq4:finite-scale} imply
$$\abs{ \Lan{n}{1} (B_1) - \Lan{n}{1} (B_2) }\le
 \int_{X} \ \abs{ \frac{1}{n} \log \norm{B_1^{(n)}(x)} - \frac{1}{n}\log \norm{B_2^{(n)}(x)}  }  \, \mu (d x) 
 < \mes_n^{1/2} \,,$$
 which proves \eqref{eq:finite-scale}.
  \end{proof}

\section{The inductive step procedure}
\label{inductive_step}
\newcommand{\realvarep}{\epsilon_{ap}}
\newcommand{\realvarka}{\varkappa_{ap}}

In this section we derive the main technical result used to prove our continuity theorem, an inductive tool based on the avalanche principle~\ref{AP-practical}, the uniform fiber-LDT~\ref{def:unif:fiber-LDT} and the nearly uniform upper semicontinuity~\ref{n-unif-usc}. 

\medskip

All estimates involving two consecutive scales $n_0$, $n_1$ of the inductive procedure will cary errors of order at most $\frac{n_0}{n_1}$. We begin with a simple lemma which shows that we may always assume that $n_1$ is a multiple of $n_0$, otherwise an extra error term of the same order is accrued. 


\begin{lemma}\label{lemma:scales-divide}
Let $A \in \cocycles$ be an $L^1$-bounded cocycle, and let $C$ be its $L^1$-bound.  
If $n_0, n_1, n, r \in \N$ are such that $n_1 = n \cdot n_0 + r$ and $0 \le r \le n_0$, then
\begin{equation}\label{eq:scales-divide}
- 2 C \, \frac{n_0}{n_1} + \Lan{(n+1) \, n_0}{1} (A) \le  \Lan{n_1}{1} (A) \le  \Lan{n \, n_0}{1} (A)  +  2 C \, \frac{n_0}{n_1}
\end{equation}
\end{lemma}

\begin{proof} From the $L^1$-boundedness assumption on $A$, for all $m \ge 1$, 
$\abs{ \Lan{m}{1} (A) } \le C$.

Since $n_1 = n \cdot n_0 + r$ and $r \ge 0$,  $\An{n_1} (x) = \An{r} (T^{n \, n_0} x) \cdot \An{n \, n_0} (x)$, hence  
$$\norm{A^{(n_1)}(x)} \leq \norm{A^{(r)}(\transl^{n\,n_0} x)} \, \norm{A^{(n\,n_0)}(x)}\,.$$
Taking logarithms, dividing by $n_1$ then integrating in $x$ we get:
$$
\Lan{n_1}{1} (A) \le \frac{n \, n_0}{n_1} \, \Lan{n \, n_0}{1} (A) + \frac{r}{n_1} \, \Lan{r}{1} (A)\,.
$$
This implies
$$
\Lan{n_1}{1} (A) -  \Lan{n \, n_0}{1} (A)  \le \frac{r}{n_1} \, [  \Lan{r}{1} (A) -  \Lan{n \, n_0}{1} (A) ] \le 2 C \, \frac{r}{n_1}\,,
$$
which proves the right hand side of \eqref{eq:scales-divide}.

Now write $(n+1) \, n_0 = n_1 + q$, where $q = n_0 - r$, so $0 \le q \le n_0$. Then
\begin{align*}
\An{(n+1) \, n_0} (x) & = \An{q} (T^{n_1} x) \cdot \An{n_1} (x)\,,\\
\norm{ \An{(n+1) \, n_0} (x)}  & \le \norm{\An{q} (T^{n_1} x) } \,  \norm{ \An{n_1} (x) }\,. 
\end{align*}
Taking logarithms, dividing by $(n+1) \, n_0$ then integrating in $x$ we get:
$$
\Lan{(n+1) \, n_0}{1} (A) \le \frac{n_1}{(n+1) \, n_0} \, \Lan{n_1}{1} (A) + \frac{q}{(n+1) \, n_0} \, \Lan{q}{1} (A)\,.
$$

This implies
$$
\Lan{(n+1) \, n_0}{1} (A)  -  \Lan{n_1}{1} (A)  \le \frac{q}{(n+1) \, n_0} \, [ \Lan{q}{1} (A) - \Lan{n_1}{1} (A) ] \le 2 C \, \frac{1}{(n+1)}\,, 
$$
which proves the left hand side of \eqref{eq:scales-divide}.
\end{proof}


\begin{lemma}\label{lemma:angle:average:pointwise}
Let $B \in \cocycles$ satisfying a fiber-LDT with parameter $\param = (\nzerobar, \devf, \mesf) \in \params$. Let $m_1, m_2, n \in \N$ and $\eta > 0$ be such that $m_i \ge n \ge \nzerobar$  for $i = 1, 2$ and
\begin{equation*}
\abs{ \Lan{m_2 + m_1}{1} (B) - \Lan{m_i}{1} (B) } < \eta
\end{equation*} 
Then
\begin{equation}\label{eq:angle:average:pointwise}
\frac{\norm{\Bn{m_2+m_1} (x)}}{\norm{\Bn{m_2} (\transl^{m_1} x)} \cdot \norm{\Bn{m_1} (x)}} > e^{- (m_1+m_2) ( \eta + 2 \dev_n)}
\end{equation}
for all $x$ outside a set of measure $ < 3 \, \mes_n$.
\end{lemma}

\begin{proof}
Applying (one side inequality in) the fiber-LDT to the cocycle $B$ at scale $m_2+m_1$, for all $x$ outside a set of measure $ < \mes_{m_2+m_1} < \mes_n$, we have:
$$ \frac{1}{m_2+m_1} \, \log \norm{ \Bn{m_2+m_1} (x) } >  \Lan{m_2+m_1}{1} (B) - \dev_{m_2+m_1} \ge \Lan{m_2+m_1}{1} (B) - \dev_{n} \,,$$
hence
\begin{equation}\label{eq1:angle}
\norm{\Bn{m_2+m_1} (x) } >  e^{ (m_1 + m_2)\, \Lan{m_1+m_2}{1} (B)  - (m_2+m_1) \, \dev_n} \,.  
\end{equation}

Applying (the other side inequality in) the fiber-LDT to the cocycle $B$ at scales $m_2, m_1$, for all $x$ outside a set of measure $ < \mes_{m_2} + \mes_{m_1} < 2 \mes_n$,
\begin{align*}
 \frac{1}{m_1} \, \log \norm{ \Bn{m_1} (x) } & < \Lan{m_1}{1} (B) + \dev_{m_1} < \Lan{m_1}{1} (B) + \dev_{n}\,,\\
 \frac{1}{m_2} \, \log \norm{ \Bn{m_2} (T^{m_1} x) } & < \Lan{m_2}{1} (B) + \dev_{m_2} < \Lan{m_2}{1} (B) + \dev_{n}\,.
 \end{align*}
Thus
\begin{align} 
\norm{\Bn{m_1} (x) } < e^{m_1 \Lan{m_1}{1} (B) + m_1 \dev_n}   \label{eq2:angle}\,, \\
\norm{\Bn{m_2} (T^{m_1} x) } < e^{m_2 \Lan{m_2}{1} (B) + m_2 \dev_n}  \label{eq3:angle}\,. 
\end{align}

Combining \eqref{eq1:angle}, \eqref{eq2:angle} and \eqref{eq3:angle}, for $x$ outside a set of measure $ < 3 \, \mes_n$ we get:
\begin{align*}
& \frac{\norm{\Bn{m_2+m_1} (x)}}{\norm{\Bn{m_2} (\transl^{m_1} x)} \cdot \norm{\Bn{m_1} (x)}} \\
& \qquad >  e^{m_1 ( \Lan{m_2+m_1}{1} (B) - \Lan{m_1}{1} (B)  ) + m_2 ( \Lan{m_2+m_1}{1} (B) - \Lan{m_2}{1} (B) )  - 2 (m_2+m_1) \, \dev_n} \\
& \qquad >  e^{- (m_1+m_2) ( \eta + 2 \dev_n)}\,,
\end{align*}
which proves the lemma.
\end{proof}

\begin{remark}\label{remark:angle:average:pointwise}\normalfont
We note that all is needed in the proof of  \eqref{eq:angle:average:pointwise} is the availability of the fiber-LDT estimate precisely at scales $m_1$, $m_2$ and $m_1 + m_2$ and not at all scales $n \ge \nzerobar$. This is of course irrelevant here, but it will be helpful in other contexts, when the (full) fiber-LDT estimate is not available a-priori.   
\end{remark}

\begin{proposition}[inductive step procedure]\label{inductive-step}
Let $A \in \cocycles_m$ be a measurable cocycle such that $\gapc (A) := L_1 (A) - L_2 (A) > 0$. Fix $0 < \ep < \gapc (A) / 20$.

There are $C = C(A) > 0$, $\delta = \delta (A, \ep) > 0$, $\nzz = \nzz (A, \ep) \in \N$, $\mesf = \mesf (A, \ep) \in \mesfs$ such that for any $n_0 \ge \nzz$, if the inequalities
\begin{align}
\rm{(a)} \quad & \Lan{n_0}{1} (B) - \Lan{2 n_0}{1} (B) < \eta_0 \label{eqa:inductive-step}\\
\rm{(b)} \quad & \abs{\Lan{n_0}{1} (B) - \Lan{n_0}{1} (A)} < \theta_0 \label{eqb:inductive-step}
\end{align}
hold for a cocycle $B \in \cocycles_m$ with $\dist (B, A) < \delta$, and if the positive numbers $\eta_0, \theta_0$ satisfy
\begin{equation}\label{cond:eta-theta-ep}
4 \eta_0 + 2 \theta_0 < \gapc (A) - 12 \ep\,,
\end{equation}
then for any integer $n_1$ such that
\begin{equation}\label{cond:n0-n1}
n_0^{1+} \le n_1 \le n_0 \cdot \mes_{n_0}^{-1/2}\,,
\end{equation}
we have:
\begin{align}
\abs{ \   \Lan{n_1}{1} (B)  + \Lan{n_0}{1} (B) - 2 \Lan{2 n_0}{1} (B)   \ }  & < C \frac{n_0}{n_1}\,. \label{topLE:n1-n0} 
\end{align}

Furthermore,
\begin{align}
 \rm{(a\!+\!\! +)}   \quad & \Lan{n_1}{1} (B) - \Lan{2 n_1}{1} (B) < \eta_1 \label{eqa++:inductive-step}\\
 \rm{(b\!+\!\! +)}   \quad &  \abs{\Lan{n_1}{1} (B) - \Lan{n_1}{1} (A)} < \theta_1\,, \label{eqb++:inductive-step}
\end{align}
where
\begin{align}
\theta_1 & = \theta_0 + 4 \eta_0 + C \, \frac{n_0}{n_1} \,,\label{theta1}\\
\eta_1 & = C \, \frac{n_0}{n_1} \,.\label{eta1}
\end{align}
\end{proposition}


\begin{proof}
Since $L_1(A) > L_2 (A)$, the cocycle $A$ satisfies a uniform fiber-LDT. Moreover, Lemma~\ref{lemma:usc-gapn} also applies.

Pick $\delta = \delta (A, \ep) > 0$, $\nzz = \nzz (A, \ep) \in \N$, $\mesf = \mesf (A, \ep) \in \mesfs$ such that for any $n \ge \nzz$, both the uniform fiber-LDT and Lemma~\ref{lemma:usc-gapn} apply for all cocycles $B \in \cocycles_m$ with $\dist (B, A) < \delta$, for all $n \ge \nzz$ and for all $x$ outside a set of measure $ < \mes_n$. 

Let $n_0 \ge \nzz$ and assume \eqref{eqa:inductive-step} and \eqref{eqb:inductive-step} hold for all cocycles $B \in \cocycles_m$ with $\dist (B, A) < \delta$.

From \eqref{eqa:inductive-step}, applying Lemma~\ref{lemma:angle:average:pointwise}, we have:
\begin{equation}\label{eq:angle-1}
\frac{\norm{\Bn{2 n_0} (x)}}{\norm{\Bn{n_0} (\transl^{n_0} x)} \cdot \norm{\Bn{n_0} (x)}} >  e^{- n_0 (2 \eta_0 + 4 \dev_{n_0})} \ge e^{- n_0 (2 \eta_0 + 4 \ep)}   =: \realvarep
\end{equation}
for all $x$ outside a set of measure $ < 3 \mes_n$. 


This estimate will ensure that the angles condition in the avalanche principle (Proposition~\ref{AP-practical}) holds. Moreover, due to the assumption~\eqref{eqb:inductive-step}, applying Lemma~\ref{lemma:usc-gapn}, for $x$ outside a set of measure $ <  \mes_n$, we have:
\begin{equation}\label{eq:gap-1}
 \rgap (\Bn{n_0} (x)) > e^{n_0 \, (\gapc (A) - 2 \theta_0 - 3 \ep)} =: \frac{1}{\realvarka}\,,
\end{equation}
which will ensure that the gaps condition in the avalanche principle  also holds.

Let $\B_{n_0}$ be the union of the exceptional sets in \eqref{eq:angle-1} and \eqref{eq:gap-1}. To simplify notations, replace the deviation set measure function $\mesf$ by $2 \, \mesf$, so we may assume $\mu (\B_{n_0} ) < \mes_{n_0}$ (we will tacitly do this throughout the paper).

Let $n_1$ be an integer such that $n_0^{1+} \le n_1 \le n_0 \cdot \mes_{n_0}^{-1/2}$. Since $\mesf (t)$ decreases at least like  $t^{- c}$ (as $t \to \infty$) for some $c>0$, and since $\mesf$ depends on $\ep$ and $A$, $\nzz$ might need to be chosen larger, depending on $\ep$ and $A$ so that if $n_0 \ge \nzz$ then the integer interval  $[ n_0^{1+}, n_0 \cdot \mes_{n_0}^{-1/2} ] $ is large enough. 

Moreover, due to Lemma~\ref{lemma:scales-divide} we may assume that $n_1 = n \cdot n_0$ for some $n \in \N$.  To see this, note that once  \eqref{topLE:n1-n0} is proven for scales that are multiples of $n_0$, in particular for the scales $n_1'=n\,n_0$ and $n_1''=(n+1)\,n_0$, then using~\eqref{eq:scales-divide} we derive \eqref{topLE:n1-n0} for any scale $n_1$ such that $n\,n_0 \le n_1 \le (n+1)\,n_0$. Furthermore, \eqref{eqa++:inductive-step} and \eqref{eqb++:inductive-step} will be  derived directly from \eqref{topLE:n1-n0}.

For every $ 0 \le i \le n-1$ define $$g_i = g_i (x) := \Bn{n_0} (\transl^{i \, n_0} \, x)\,.$$
Then clearly $g^{(n)} = \Bn{n_1} (x)$ and $g_i \cdot g_{i-1} =  \Bn{2 n_0} (\transl^{(i-1) n_0}  x) $ for all $1 \le i \le n-1$.

Let $\Bbar_{n_0} := \bigcup_{i=0}^{n-1} \, \transl^{- i \, n_0} \, \B_{n_0}$, so
$\mu (\Bbar_{n_0} ) < n \ \mes_{n_0}$ and if $x \notin \Bbar_{n_0}$ then 
\begin{align*}
\rgap (g_i) >  \frac{1}{\realvarka} & \  \text{ for all }  \ \  0 \le i \le n-1\,,  
\\
 \frac{\norm{g_i \cdot g_{i-1}}}{\norm{g_i} \cdot \norm{g_{i-1}}}  >  \realvarep  & \
 \text{ for all } \  \ 1 \le i \le n-1\,.  
\end{align*}

Note also that condition~\eqref{cond:eta-theta-ep} implies $\realvarka \ll \realvarep^2$.

Therefore, we can apply the avalanche principle (Proposition~\ref{AP-practical}) and obtain:
\begin{align*} 
\abs{ \ \log \norm{g^{(n)}} + \sum_{i=1}^{n-2} \log \norm{g_i} -  \sum_{i=1}^{n-1} \log \norm{g_i \cdot g_{i-1}} \ } & \less n \cdot \frac{\realvarka}{\realvarep^2} \,.
\end{align*}

Note that
$\ \displaystyle \frac{\realvarka}{\realvarep^2} = e^{-n_0 \ (\gapc (A) - 4 \eta_0 - 2 \theta_0 - 11 \ep)} < e^{- \ep \, n_0}$.

Since the deviation set measure functions $\mesf \in \mesfs$ decay at most exponentially fast, we may assume that $e^{- \ep t } \le \mesf (t)$ for $t \ge \nzz$. Hence we have $\frac{\realvarka}{\realvarep^2}  < \mes_{n_0}$.

The AP applied to our data implies that for all $x \notin \Bbar_{n_0}$, where $\mu (\Bbar_{n_0}) < n \mes_{n_0}$,
\begin{align}
&  \Bigl|  \log \norm{\Bn{n_1} (x) } + \sum_{i=1}^{n-2} \, \log  \norm{  \Bn{n_0} (\transl^{i n_0} \, x) }  
 \notag \\
  & \qquad   \qquad 
 - \sum_{i=1}^{n-1} \, \log \, \norm{\Bn{2 n_0} (\transl^{(i-1) n_0} \, x) }  \Bigr| \less n \ \mes_{n_0}\,. & \label{IST:LE-n_0} 
\end{align}
Divide both sides of \eqref{IST:LE-n_0}  by $n_1 = n \cdot n_0$ to get:
\begin{align*}
&\Bigl| \frac{1}{n_1} \, \log \norm{\Bn{n_1} (x) } + \frac{1}{n} \, \sum_{i=1}^{n-2} \, \frac{1}{n_0} \, \log \norm{\Bn{n_0} (\transl^{i n_0} \, x) }  \\ 
& \qquad - \frac{2}{n} \, \sum_{i=1}^{n-1} \, \log \, \frac{1}{2 n_0} \, \norm{\Bn{2 n_0} (\transl^{(i-1) n_0} \, x) }  \Bigr|  \less  \mes_{n_0} \,. 
\end{align*}
Integrating in $x$ we conclude:
\begin{align*}
  \abs{ \Lan{n_1}{1} (B)  +   \frac{n-2}{n} \, \Lan{n_0}{1} (B) - \frac{2 (n-1)}{n} \Lan{2 n_0}{1} (B) } & \\
  \qquad  \qquad \less \mes_{n_0}  + C(A) \, \mes_{n_0} <   C  \, \mes_{n_0} \,.&
\end{align*}
The term on the left hand side of the above inequality can be written in the form
$$\abs{ \ \Lan{n_1}{1} (B)  + \Lan{n_0}{1} (B) - 2 \Lan{2 n_0}{1} (B) - \frac{2}{n} \, [ \Lan{n_0}{1} (B) - \Lan{2 n_0}{1} (B) ]  \  }\,,$$
hence we conclude:
\begin{align}
& \abs{ \   \Lan{n_1}{1} (B)  + \Lan{n_0}{1} (B) - 2 \Lan{2 n_0}{1} (B)   \ } \notag \\
  & \qquad  \qquad 
<   C \,  \mes_{n_0}  + \frac{2}{n} [ \Lan{n_0}{1} (B) - \Lan{2 n_0}{1} (B) ]  < C \frac{n_0}{n_1} \,. \label{eq!:indstep}
\end{align} 

Clearly the same argument leading to  \eqref{IST:LE-n_0} will hold for $2 n_1$ instead of $n_1$, which via the triangle inequality proves \eqref{eqa++:inductive-step}. 

We can rewrite \eqref{eq!:indstep} in the form
\begin{equation}\label{eq*:indstep}
\abs{ \   \Lan{n_1}{1} (B)  - \Lan{n_0}{1} (B) + 2 [ \Lan{n_0}{1} (B) - \Lan{2 n_0}{1} (B) ]   \ }  < C \frac{n_0}{n_1}\,.
\end{equation} 

Using \eqref{eq*:indstep} for $B$ and $A$ we get:
\begin{align*}
&\abs{ \Lan{n_1}{1}  (B) - \Lan{n_1}{1} (A) }   \\
&  \qquad < \abs{ \   \Lan{n_1}{1} (B)  - \Lan{n_0}{1} (B) + 2 [ \Lan{n_0}{1} (B) - \Lan{2 n_0}{1} (B) ]   \ }  \\ 
&  \qquad + \abs{ \   \Lan{n_1}{1} (A)  - \Lan{n_0}{1} (A) + 2 [ \Lan{n_0}{1} (A) - \Lan{2 n_0}{1} (A) ]   \ }   \\
&  \qquad + \abs{ \Lan{n_0}{1}  (B) - \Lan{n_0}{1} (A) } \\
& \qquad + 2 \abs{\Lan{n_0}{1} (B) - \Lan{2 n_0}{1} (B) }  + 2 \abs{\Lan{n_0}{1} (A) - \Lan{2 n_0}{1} (A) }  \\
& \qquad < \theta_0 + 4 \eta_0 + C \frac{n_0}{n_1}\,.  \qquad 
\end{align*}
 \end{proof}

\section{General continuity theorem}
\label{general_cont_thm}
We are now ready to prove our abstract  continuity result for Lyapunov exponents of linear cocycles.

\begin{theorem}\label{thm:general-cont}
Let $A \in \cocycles_m$ be a measurable cocycle for which $L_1 (A) > L_2 (A)$. Then the map $\cocycles_m \ni B \mapsto L_1 (B)$ is continuous at $A$ and the map $\cocycles_m \ni B \mapsto L_1 (B) - L_2 (B)$ is lower semicontinuous at $A$.
\end{theorem}

\begin{proof}
Let $0 < \ep < \gapc (A)/100$ be arbitrary but fixed. 

Since $\Lan{n}{1} (A) \to L_1 (A)$ as $n \to \infty$, there is $n_{0 2} = n_{0 2} (A, \ep) \in \N$ such that for all $n \ge n_{0 2}$ we have 
\begin{equation}\label{eq1:general-cont}
\Lan{n}{1} (A) - \Lan{2 n}{1}  (A) < \ep\,.
\end{equation}

We will apply the inductive step Proposition~\ref{inductive-step} repeatedly. 
We first choose the relevant parameters (which will depend on $A$ and $\ep$) so that both the inductive step Proposition~\ref{inductive-step} and the finite scale continuity Proposition~\ref{prop:finite-scale} apply. The latter will ensure that the assumptions \eqref{eqa:inductive-step}, \eqref{eqb:inductive-step} and \eqref{cond:eta-theta-ep} of the inductive step Proposition~\ref{inductive-step} are satisfied for a large enough scale $n_0 = n_0 (A, \ep)$, so we can start running the inductive argument with that scale.

Let $\mesf \in \mesfs$ be the sum of the corresponding deviation measure functions in the inductive step Proposition~\ref{inductive-step} and the finite scale continuity Proposition~\ref{prop:finite-scale}. 

Let $\delta_0$ be less than the size of the neighborhood of $A \in \cocycles_m$ from the inductive step Proposition~\ref{inductive-step} and from the finite scale continuity Proposition~\ref{prop:finite-scale} respectively. Let $C_1$ be the constant in the finite scale continuity Proposition~\ref{prop:finite-scale} and let $C$ be the constant in  the inductive step Proposition~\ref{inductive-step}.

Finally, let the scale  $n_0$ be greater than the thresholds $\nzz$ from the inductive step Proposition~\ref{inductive-step}, $\nzo$ from the finite scale continuity Proposition~\ref{prop:finite-scale} and $n_{0 2}$ from \eqref{eq1:general-cont} above. Moreover, assume $n_0$ to be large enough so that $e^{- C_1 \, 2 n_0} < \delta_0$, \ $\mes_{n_0}^{1/2} < \ep$ and $n^{0 +} \ll \mes_n^{ - 1/2}$ for $n \ge n_0$ and $C \, n_0^{- 0 -} \less \ep$.

 Let $\delta := e^{- C_1 \, 2 n_0}$ and let $B \in \cocycles_m$  with $\dist (B, A) < \delta$.
 
 Since $\delta = e^{- C_1 \, 2 n_0} < e^{- C_1 \, n_0}$, we can apply the finite scale continuity Proposition~\ref{prop:finite-scale} (with $B_2 = B$ and $B_1 = A$) at scales $2 n_0$ and $n_0$ and get:
 \begin{align}
 \abs{ \Lan{n_0}{1}  (B) - \Lan{n_0}{0} (A) }  & < \mes_{n_0}^{1/2} =: \theta_0 < \ep\,,  \label{eq2:general-cont}\\
 \abs{ \Lan{2 n_0}{1}  (B) - \Lan{2 n_0}{0} (A) }  & < \mes_{2 n_0}^{1/2} < \mes_{n_0}^{1/2} = \theta_0\,.  \label{eq3:general-cont}\
\end{align}

Then \eqref{eq1:general-cont}, \eqref{eq2:general-cont}, \eqref{eq3:general-cont} imply
\begin{equation}\label{eq4:general-cont}
\Lan{n_0}{1} (B) - \Lan{2 n}{1} (A) < 2 \, \mes_{n_0}^{1/2} + \ep =: \eta_0 < 3 \ep\,.
\end{equation}

The inequalities \eqref{eq4:general-cont} and \eqref{eq3:general-cont} imply the assumptions \eqref{eqa:inductive-step} and \eqref{eqb:inductive-step} in the inductive step Proposition~\ref{inductive-step}, and since  
$$ \displaystyle 2 \theta_0 + 4 \eta_0 < 2 \ep + 12 \ep = 14 \ep < \gapc (A) - 12 \ep \,,$$
the condition \eqref{cond:eta-theta-ep} between parameters is also satisfied. 

We can apply the inductive step Proposition~\ref{inductive-step} and conclude that 
\begin{equation}\label{gap:n0:general-cont}
\Lan{n_0}{1} (B) - \Lan{n_0}{2} (B) > (\gapc (A) - \theta_0 - 2 \ep) \cdot (1- \mes_{n_0})\,.
\end{equation}
Then say for $n_1 \asymp n_0^{1+}$ we have:
\begin{align}
\Lan{n_1}{1} (B) - \Lan{2 n_1}{1} (B) < \eta_1\,, \label{eqa1:general-cont}\\
\abs{\Lan{n_1}{1} (B) - \Lan{n_1}{1} (A)} < \theta_1\,, \label{eqb1:general-cont}
\end{align}
where
\begin{align}
\theta_1 & = \theta_0 + 4 \eta_0 + C \, \frac{n_0}{n_1}\,, \label{theta1:general-cont}\\
\eta_1 & = C \, \frac{n_0}{n_1}\,. \label{eta1:general-cont}
\end{align}

But 
$$\displaystyle 2 \theta_1 + 4 \eta_1 = (2 \theta_0 + 8 \eta_0) + 4 C  \, \frac{n_0}{n_1} < (2 \ep + 24 \ep) + \ep < \gapc (A) - 12 \ep \,,$$ 
hence the inductive step Proposition~\ref{inductive-step} applies again and we get:
\begin{equation}\label{gap:n1:general-cont}
\Lan{n_1}{1} (B) - \Lan{n_1}{2} (B) > (\gapc (A) - \theta_1 - 2 \ep) \cdot (1- \mes_{n_1})\,.
\end{equation}
Then say for $n_2 \asymp n_1^{1+}$ we have:
\begin{align}
\Lan{n_2}{1} (B) - \Lan{2 n_2}{1} (B) < \eta_2\,, \label{eqa2:general-cont}\\
\abs{\Lan{n_2}{1} (B) - \Lan{n_2}{1} (A)} < \theta_2\,, \label{eqb2:general-cont}
\end{align}
where
\begin{align}
\theta_2 & = \theta_1 + 4 \eta_1 + C \, \frac{n_1}{n_2}\,, \label{theta2:general-cont}\\
\eta_2 & = C \, \frac{n_1}{n_2}\,. \label{eta2:general-cont}
\end{align}

Moreover, 
$$\displaystyle 2 \theta_2 + 4 \eta_2 = (2 \theta_0 + 8 \eta_0) + [ 10 C  \, \frac{n_0}{n_1}  + 6 C \, \frac{n_1}{n_2} ] < (2 \ep + 24 \ep) + \ep < \gapc (A) - 12 \ep\,.$$

It is now clear how we continue this procedure. Going  from step $k$ to step $k+1$, we choose a scale $n_{k+1} \asymp n_k^{1 +}$ and we have:
 \begin{equation}\label{gap:nk:general-cont}
\Lan{n_k}{1} (B) - \Lan{n_k}{2} (B) > (\gapc (A) - \theta_k - 2 \ep) \cdot (1- \mes_{n_k})
\end{equation}
and 
\begin{align}
\Lan{n_{k+1}}{1} (B) - \Lan{2 n_{k+1}}{1} (B) < \eta_{k+1}\,, \label{eqak:general-cont}\\
\abs{\Lan{n_{k+1}}{1} (B) - \Lan{n_{k+1}}{1} (A)} < \theta_{k+1}\,, \label{eqbk:general-cont}
\end{align}
where
\begin{equation}\label{etak:general-cont}
\eta_{k+1} = C \, \frac{n_k}{n_{k+1}} 
\end{equation}
and
\begin{align*}
\theta_{k+1} & = \theta_k + 4 \eta_k + C \, \frac{n_1}{n_2} \\
& = (\theta_0 + 4 \eta_0) + 5 C \, \sum_{i=0}^{k-1} \, \frac{n_i}{n_{i+1}} + C \, \frac{n_k}{n_{k+1}}  \\
& <   (\theta_0 + 4 \eta_0) + 5 C \, \sum_{i=0}^{\infty} \, \frac{n_i}{n_{i+1}}  \\
& < (\theta_0 + 4 \eta_0) + 10 C \, n_0^{ - 0 -} < (\ep + 12 \ep) + 10 \ep = 23 \, \ep\,.
\end{align*}
Hence
\begin{equation}\label{thetak:general-cont}
\theta_{k+1} < 23 \, \ep\,.
\end{equation}

Moreover
\begin{align*}
2 \theta_{k+1} + 8 \eta_{k+1} & = (2 \theta_0 + 8 \eta_0) + 10 C \, \sum_{i=0}^{k-1} \, \frac{n_i}{n_{i+1}} + 6 C \, \frac{n_k}{n_{k+1}} \\
& < (2 \theta_0 + 8 \eta_0) + 10 C \, \sum_{i=0}^{\infty} \, \frac{n_i}{n_{i+1}} \\
& < (2 \theta_0 + 8 \eta_0) + 20 C \, n_0^{ - 0 -} < (2 \ep + 24 \ep) + 20 \ep = 46 \ep\,,
\end{align*}
so
\ $\displaystyle  \theta_{k+1} + 8 \eta_{k+1}  < \gapc (A) - 12 \ep$,
ensuring that the inductive process runs indefinitely.

Now take the limit as $k \to \infty$ in \eqref{eqbk:general-cont}, and using \eqref{thetak:general-cont} we have
$$\abs{ L_1 (B) - L_1 (A) } \le 23 \ep\,,$$
which proves the continuity at  $A$ of the top Lyapunov exponent $L_1$.

Moreover, taking the limit as $k \to \infty$ in \eqref{gap:nk:general-cont}, and using again \eqref{thetak:general-cont} we have
$$L_1 (B) - L_2 (B) \ge \gapc (A) - 23 \ep - 2 \ep = L_1 (A) - L_2 (A) - 25 \ep\,,$$
which proves the lower semicontinuity at $A$ of the gap between the first two LE.
 \end{proof}


Note that estimate \eqref{eqbk:general-cont} in the proof of Theorem~\ref{thm:general-cont} says that if the cocycle $B$ is close enough to $A$, then
\begin{equation}\label{eq*:general-cont-remark}
\abs{ \Lan{n}{1} (B) - \Lan{n}{1} (A)  } \less \ep
\end{equation} 
holds for an increasing sequence of scales $n = n_{k+1}, \, k \ge 0$.

In fact, a slight modification of the argument shows that \eqref{eq*:general-cont-remark} holds in fact for {\em all} large enough scales $n$.

Indeed, it is enough to first ensure that the base step of the inductive procedure, i.e. that the estimates 
 \begin{align*}
 \abs{ \Lan{n_0}{1}  (B) - \Lan{n_0}{0} (A) }  & < \mes_{n_0}^{1/2} =: \theta_0 < \ep  \\
 \abs{ \Lan{2 n_0}{1}  (B) - \Lan{2 n_0}{0} (A) }  & < \mes_{2 n_0}^{1/2} < \mes_{n_0}^{1/2} = \theta_0  
\end{align*}
hold not just for a single scale $n_0$, but for a whole (finite) interval of scales $\scale_0 = [ \nzz, e^{\nzz} ] =: [ n_0^-, n_0^+ ]$, where $\nzz$ is greater than the applicability threshold 
of various estimates (e.g. unform fiber-LDT, finite scale continuity etc). 

Let $\psi (t) := t^{1+}$ and define inductively the intervals of scales $\scale_1 = [ \psi (n_0^-), \psi (n_0^+) ] =: [ n_1^-, n_1^+ ]$, \ 
$\scale_{k+1} = [ \psi (n_k^-), \psi (n_k^+) ] =: [ n_{k+1}^-, n_{k+1}^+ ]$ for all $k \ge 0$.

It follows that if $n = n_{1} \in \scale_{1}$, then $n = n_1 \asymp \psi (n_0) = n_0^{1+}$ for some $n_0 \in \scale_0$, and so \eqref{eqa1:general-cont} and \eqref{eqb1:general-cont} hold for {\em all} $n_1 \in \scale_1$.

Continuing inductively, for every $k \ge 1$, if $n = n_{k+1} \in \scale_{k+1}$, then there is $n_k \in \scale_k$ such that  $n = n_{k+1} \asymp \psi (n_k) = n_k^{1+}$ and then \eqref{eqak:general-cont} and \eqref{eqbk:general-cont} hold as well.

The intervals $\scale_0$ and $\scale_1$ overlap because
$$n_1^- = \psi (n_0^-) = \nzz^{1+} < e^{\nzz} = n_0^+ \,.$$
Then since $\psi$ is increasing and $\scale_{k+1} \asymp \psi (\scale_k)$, the intervals $\scale_k$ and $\scale_{k+1}$ will overlap for all $k \ge 0$.

Therefore, if $n \ge n_1^-$ then $n \in \scale_{k+1}$ for some $k \ge 0$ and so \eqref{eq*:general-cont-remark} holds. This means, moreover, that we may apply lemma~\ref{lemma:usc-gapn} at all such scales and conclude that for all $x$ outside a set of measure $ < \mes_n$,
$$\frac{1}{n} \, \log \rgap (\Bn{n} (x)) > \gapc (A) - 5 \ep \,.$$

We conclude that the following uniform, finite scale statement holds.

\begin{lemma}\label{gen-theorem-extension}
Given a cocycle $A \in \cocycles_m$ with $L_1 (A) > L_2 (A)$ and $0 < \ep < \gapc (A) /100$, there are $\delta = \delta (A, \ep) > 0$ and $n_0 = n_0 (A, \ep) \in \N$, such that for all $n \ge n_0$ and for all $B \in \cocycles_m$ with $\dist (B, A) < \delta$ we have:
\begin{align}
\abs{ \Lan{n}{1} (B) - \Lan{n}{1} (A) } < \ep\,,    \label{eq:gen-theorem-extension}\\
\frac{1}{n} \, \log \rgap (\Bn{n} (x)) > \gapc (A) - 5 \ep\,, \label{eq:gen-theorem-extension-gaps}
\end{align}
for all $x$ outside a set of measure $ < \mes_n$.
\end{lemma}

\begin{corollary}
\label{coro:all LE cont}
For all $m \ge 1$, and for all $1\le k \le m$, the Lyapunov exponents $L_k \colon \cocycles_m \to [- \infty, \infty)$ are continuous functions.
\end{corollary}

\begin{proof}
Let $A \in \cocycles$ be a measurable cocycle. If $L_1 (A) > L_2 (A)$, then we can conclude, from Theorem~\ref{thm:general-cont} above that $L_1$ is continuous at $A$ and that $L_1 - L_2$ is lower semicontinuous at $A$. 

If $A$ has a different gap pattern, by taking appropriate exterior powers, we can always reduce the problem to one where there is a gap between the first two Lyapunov exponents. 

For instance, if $L_1 (A) = L_2 (A) > L_3 (A) \ge \ldots L_m (A)$, consider instead the cocycle $\wedge_2 A$. Clearly $L_1 (\wedge_2 A) = L_1 (A) + L_2 (A)$ and $L_2 (\wedge_2 A) = L_1 (A) + L_3 (A)$, hence $L_1 (\wedge_2 A) - L_2 (\wedge_2 A)  = L_2 (A) - L_3 (A) > 0$. Then there is a gap between the first two Lyapunov exponents of $\wedge_2 A$. This implies, using Theorem~\ref{thm:general-cont} for $\wedge_2 A$, that the block $L_1 + L_2$ is continuous and the gap $L_2 - L_3$ is lower semi-continuous at $A$.

This argument shows that given a cocycle $A \in \cocycles$ with any gap pattern, the corresponding Lyapunov blocks are all continuous at $A$, while the corresponding gaps are lower semicontinuous at $A$.

Moreover, the general assumptions made on the space of cocycles ensure that the map 
$$\cocycles_m \ni B \mapsto L_1 (B) + \ldots + L_m (B) = \int_X \, \log \abs{\det [B (x) ]} \, \mu (d x)$$
is continuous everywhere.

It is then a simple exercise (see Lemma 6.1 and Theorem 6.2 in \cite{DK2} for its solution) to see that this is all that is needed to conclude continuity of each individual Lyapunov exponent, irrespective of any gap pattern.
 \end{proof}

\section{Modulus of continuity}
\label{modulus_cont}
The following proposition, which is also interesting in itself, will be the main ingredient in obtaining the modulus of continuity of the top Lyapunov exponent. It gives the rate of convergence of the finite scale exponents $\Lan{n}{1} (B)$ to the top Lyapunov exponent $L_1 (B)$ and it gives an estimate on the proximity of these finite scale exponents at different scales. 

These estimates are uniform in a neighborhood of a cocycle $A \in \cocycles_m$ for which $L_1 (A) > L_2 (A)$, and they depend on a deviation measure function $\mesf = \mesf (A) \in \mesfs$, which will be fixed in the beginning. Define the map $\psi (t) = \psi_{\mesf} (t) :=  t \cdot [\mesf (t)]^{-1/2}$, and let  $\phi = \phi_{\mesf}$ be its inverse. Moreover, for every integer $n \in \N$, denote $n+\!+ := \intpart{\psi (n)} = \intpart{n \, \mes_n^{-1/2}}$ and $n-\!- := \intpart{\phi (n)}$, so $(n+\!+)-- \asymp n$.

These estimates will be obtained by applying repeatedly the inductive step Proposition~\ref{inductive-step}. In order to obtain the sharpest possible estimate, when going from one scale to the next, we will make the greatest possible jump, which is why we have defined the 'next scale' $n+\!+$ above as $\intpart{n \, \mes_n^{-1/2}}$.

\begin{proposition}[uniform speed of convergence]\label{prop:speed-conv}
Let $A \in \cocycles_m$ be a measurable cocycle for which $L_1 (A) > L_2 (A)$. There are $\delta = \delta (A) > 0$, $C = C(A)$, $\nzz = \nzz (A) \in \N$, $\mesf = \mesf (A) \in \mesfs$ such that, with the above notations,  for all $n \ge \nzz$ and for all $B \in \cocycles_m$ with $\dist (B, A) < \delta$ we have:
\begin{align}
 \Lan{n}{1} (B) - L_1 (B) & < C \, \frac{\phi (n)}{n} \le  \mes_{n-\!-}^{1/2} \label{eq:speed-conv}\\
 \abs{ \   \Lan{n+\!+}{1} (B)  + \Lan{n}{1} (B) - 2 \Lan{2 n}{1} (B)   \ }  & < C \, \frac{n}{n+\!+} \le \mes_n^{1/2} \,. \label{topLE:n-n++-} 
 \end{align}

\end{proposition}

\begin{proof}
To prove \eqref{eq:speed-conv}  it is enough to show, under similar constraints on $B$ and $n$, and for some function $\mesf = \mesf (A) \in \mesfs$, that
\begin{align}
\Lan{n}{1} (B) - \Lan{2 n}{1} (B) & < C \, \frac{\phi (n)}{n} \,. \label{eq0:speed-conv} 
\end{align}

This would imply, for all $k \ge 0$,
\begin{align}
\Lan{2^k \, n}{1} (B) - \Lan{2^{k+1} \,  n}{1} (B) & < \frac{\phi (2^k \, n)}{2^k n} \,. \label{eq1:speed-conv} 
\end{align}

Since we assume that for all deviation measure functions $\mesf \in \mesfs$, the corresponding map  $\phi = \phi_{\mesf}$ satisfies 
$$\varlimsup_{t \to \infty} \ \frac{\phi_{\mesf} (2t)}{\phi_{\mesf} (t)}  < 2 \,,$$
there is $ 0 < r < 1$ such that for large enough $n$ we have
$\phi (2 n) \le 2 r \, \phi (n)$.
Hence for all $k$ we have:
$$\frac{\phi (2^k \, n)}{2^k n} \le r^k \, \frac{\phi (n)}{n}\,.$$
We can then sum up for $k$ from $0$ to $\infty$ in \eqref{eq1:speed-conv} and derive \eqref{eq:speed-conv}.

\smallskip

To prove \eqref{eq0:speed-conv}, \eqref{topLE:n-n++-}  we follow the same procedure, based on the inductive step Proposition~\ref{inductive-step}, used in the proof the general continuity Theorem~\ref{thm:general-cont} and of its extension Lemma~\ref{gen-theorem-extension}, but with some modifications. We will work again with intervals of scales instead of individual scales. We fix $\ep := \gapc (A)/100$ so all subsequent parameters, including the deviation measure function $\mesf$, will be fixed and dependent only upon the cocycle $A$. 
Let $\mesf \in \mesfs$, $\delta_0 > 0$, $C_1 > 0$ and $C > 0$ be as in the beginning of the proof of Theorem~\ref{thm:general-cont}.

Pick $n_{0}^{-} \in \N$ large enough that for $n \ge n_{0}^{-}$ the inductive step Proposition~\ref{inductive-step} and the finite scale continuity Proposition~\ref{prop:finite-scale} apply, and that $\Lan{n}{1} (A) - \Lan{2 n}{1} (A) < \ep_0$. 

Assume also $n_{0}^{-}$ to be large enough that $e^{- C_1 \, 2 \, e^{n_{0}^{-}} } < \delta_0$, $ \mes_{n_{0}^{-}}^{1/2} < \ep_0$, $ C (n_{0}^{-})^{- 0 -} < \ep_0$, $ n^{0 +} \ll \mes_n^{-1/2}$ for $n \ge n_{0}^{-}$ and since $\mesf$ decays at most exponentially, we may also assume that $n_{0}^{-} \, \mes_{n_{0}^{-}}^{-1/2} < e^{n_{0}^{-}}$.

Now set $n_{0}^{+} := \intpart{e^{n_{0}^{-}}}, \ \scale_0 := [ n_{0}^{-}, \, n_{0}^{+} ]$ and $\delta := e^{- C_1 \, 2 n_{0}^{+}}$.

The assumptions above ensure that for all cocycles $B \in \cocycles_m$ with $\dist (B, A) < \delta$, and for all $n_0 \in \scale_0$, since $\delta =  e^{- C_1 \, 2 n_{0}^{+}} \le e^{- C_1 \, 2 n_0}  < e^{- C_1 \, n_0}$, the finite scale continuity Proposition~\ref{prop:finite-scale} applies at scales $2 n_0, \, n_0$. This implies, as in the proof of Theorem~\ref{thm:general-cont}, the assumptions in the inductive step Proposition~\ref{inductive-step} for every $n_0 \in \scale_0$.

Let $n_{1}^{-} := \intpart{\psi (n_{0}^{-})}$, \ $n_{1}^{+} := \intpart{\psi (n_{0}^{+})}$ and $\scale_1 := [ n_{1}^{-}, \, n_{1}^{+} ]  \asymp \psi (\scale_0)$.

We may assume that for every $n_1 \in \scale_1$ there is $n_0 \in \scale_0$ such that $n_1 = n_0 \, \intpart{\mes_{n_0}^{-1/2}} \asymp n_0 \, \mes_{n_0}^{-1/2} \, ( = \psi (n_0))$ (this is because by Lemma~\ref{lemma:scales-divide}, the estimates involving scales $n_1 \in \scale_1$ which are not divisible by $n_0$ will only cary  an additional error of order $\frac{n_0}{n_1}$).

We apply the inductive step Proposition~\ref{inductive-step} and obtain:
\begin{align*}
\Lan{n_1}{1} (B) - \Lan{2 n_1}{1} (B) & < C \, \frac{n_0}{n_1} \asymp C \frac{\phi (n_1)}{n_1} \,, \\
\abs{ \   \Lan{n_1}{1} (B)  + \Lan{n_0}{1} (B) - 2 \Lan{2 n_0}{1} (B)   \ }  & < C \, \frac{n_0}{n_1} \asymp C \mes_{n_0}^{1/2} \,.
\end{align*}
(Since $n_1 \asymp \psi (n_0)$, we have $\phi (n_1) \asymp n_0$, as $\phi$ is the inverse of $\psi$.)

The procedure continues in the same way, with intervals of scales defined inductively by $\scale_{k+1} = [ n_{k}^{-}, \, n_{k}^{+} ]  \asymp \psi (\scale_k)$ for all $k \ge 0$.
Again, each two consecutive intervals of scales $\scale_k$ and $\scale_{k+1}$ overlap.
Therefore, if $n \ge n_1^{-}$, then $n = n_{k+1} \in \scale_{k+1}$ for some $k \ge 0$, so there is $n_k \in \scale_k$ such that  $n_{k+1} = n_k \, \intpart{\mes_{n_k}^{-1/2}} \asymp n_k ++$. We then have:
\begin{align*}
\Lan{n_{k+1}}{1} (B) - \Lan{2 n_{k+1}}{1} (B) & <  C \, \frac{n_k}{n_{k+1}} \asymp C \,  \frac{\phi (n_{k+1})}{n_{k+1}} \,, \\
\abs{ \   \Lan{n_{k+1}}{1} (B)  + \Lan{n_k}{1} (B) - 2 \Lan{2 n_k}{1} (B)   \ }  & < C \, \frac{n_k}{n_{k+1}} \asymp C \mes_{n_k}^{1/2} \,.
\end{align*}
which completes the proof.
 \end{proof}


The following theorem shows that locally near any cocycle $A \in \cocycles_m$  for which $L_1 (A) > L_2 (A)$, the top Lyapunov exponent has a modulus of continuity given by a map that depends explicitly on a deviation measure function $\mesf$, hence on the strength of the large deviation type estimates satisfied by the dynamical system.

\begin{theorem}[modulus of continuity]\label{thm:modulus-cont}
Let $A \in \cocycles_m$ be a measurable cocycle for which $L_1 (A) > L_2 (A)$. 
There are $\delta = \delta (A) > 0$, $\mesf = \mesf (A) \in \mesfs$ and $c = c (A) > 0$ such that if we define the modulus of continuity function $\omega (h) :=  [ \mesf \ (c \, \log \frac{1}{h}) ]^{1/2}$, then for any cocycles $B_i \in \cocycles_m$ with $\dist (B_i, A) < \delta$, where $i = 1, 2$, we have:
\begin{equation}\label{eq:modulus-cont}
\abs{ L_1 (B_1) - L_1 (B_2) }  \le \omega (\dist (B_1, B_2) )  \,.
\end{equation}

More generally, if for some $1 \le k \le m$ the cocycle $A$ has the Lyapunov spectrum gap  $L_k (A) > L_{k+1} (A)$, then the map $\Lambda_k := L_1 + \ldots + L_k$ satisfies
\begin{equation}\label{eq:modulus-cont-k}
\abs{ \Lambda_k  (B_1) - \Lambda_k (B_2) }  \le \omega (\dist (B_1, B_2) )  \,.
\end{equation}
\end{theorem}

\begin{proof}
Choose parameters $\delta_0 = \delta_0 (A) > 0$, $\nzz = \nzz (A) \in \N$ and $\mesf = \mesf (A) \in \mesfs$ such that both the finite scale uniform continuity Proposition~\ref{prop:finite-scale} and the uniform speed of convergence Proposition~\ref{prop:speed-conv} apply with deviation measure function $\mesf$ for all cocycles $B \in \cocycles_m$ with $\dist (B, A) < \delta_0$ and for all $n \ge \nzz$.

Let $C_1 = C_1 (A) > 0$ be the constant from Proposition~\ref{prop:finite-scale}. 

Set $\delta := \min \{ \delta_0, \, \frac{1}{2} \, e^{- C_1 \, 4 \nzz} \}$.

Let $B_i \in \cocycles_m$ be measurable cocycles with $\dist (B_i, A) < \delta$  ($i = 1, 2$) and put $\dist (B_1, B_2) =: h  \ ( < 2 \delta \le e^{- C_1 \, 4 \nzz})$.

Set $n := \intpart{\frac{1}{2 C_1} \, \log (1/h)} \in \N$. 

Then $e^{- C_1 \, 4 n} \le h \le e^{-C_1 \, 2 n}$, so $\dist (B_1, B_2) = h \le e^{-C_1 \, 2 n}$  and $n \ge \nzz$.

\medskip

All of this preparation shows that we can apply the finite scale uniform continuity Proposition~\ref{prop:finite-scale} to $B_1, B_2$ at scales $n$ and $2 n$ and get:
\begin{align}
\abs{ \Lan{n}{1} (B_1) -  \Lan{n}{1} (B_2) } & < \mes_n^{1/2} \label{eq1:mod-cont} \,,\\
\abs{ \Lan{2 n}{1} (B_1) -  \Lan{2 n}{1} (B_2) } & < \mes_{2 n}^{1/2} < \mes_n^{1/2} \,. \label{eq2:mod-cont}
\end{align}

Since $\dist (B_i,  A) < \delta \le \delta_0$ and $n \ge \nzz$, we can also apply the uniform speed of convergence Proposition~\ref{prop:speed-conv} to $B_i$ \ ($i = 1, 2$) at scale $n$ and have:
\begin{align}
 \Lan{n+\!+}{1} (B_i) - L_1 (B_i) & <   \mes_{n}^{1/2}\label{eq3:mod-cont} \,, \\
 \abs{ \   \Lan{n+\!+}{1} (B_i)  + \Lan{n}{1} (B_i) - 2 \Lan{2 n}{1} (B_i)   \ }  & <  \mes_n^{1/2} \,. \label{eq4:mod-cont}
 \end{align}

Combining \eqref{eq1:mod-cont}, \eqref{eq2:mod-cont}, \eqref{eq3:mod-cont}, \eqref{eq4:mod-cont} we conclude:
\begin{align*}
 \abs{ L_1 (B_1) - L_1 (B_2) }  \less \mes_n^{1/2} 
 & \le [ \mesf \, ( 1 / (2 C_1)  \, \log (1/h) ) ]^{1/2} \\ 
 & =: \omega (h) = \omega (\dist (B_1, B_2)) \,. 
 \end{align*}

The more general assertion of the theorem follows by simply taking exterior powers. 
Indeed, the cocycle $\wedge_k A$ has the property 
$$L_1 (\wedge_k A) = (L_1 + \ldots + L_{k-1} + L_k) (A) > (L_1 + \ldots + L_{k-1} + L_{k+1} ) (A) = L_2 (\wedge_k A)\,,$$
hence \eqref{eq:modulus-cont-k} follows from \eqref{eq:modulus-cont} applied to $\wedge_k A$.  
\end{proof}

\bigskip

\subsection*{Acknowledgments}

The first author was supported by 
Funda\c{c}\~{a}o  para a  Ci\^{e}ncia e a Tecnologia, 
UID/MAT/04561/2013.

The second author was supported by the Norwegian Research Council project no. 213638, "Discrete Models in Mathematical Analysis".

\bigskip

\bibliographystyle{amsplain} 

\providecommand{\bysame}{\leavevmode\hbox to3em{\hrulefill}\thinspace}
\providecommand{\MR}{\relax\ifhmode\unskip\space\fi MR }
\providecommand{\MRhref}[2]{%
  \href{http://www.ams.org/mathscinet-getitem?mr=#1}{#2}
}
\providecommand{\href}[2]{#2}

\end{document}